\documentclass[a4paper,USenglish,cleveref, autoref, thm-restate]{compgeomtop_for_arxiv}

\title{Representing two-parameter persistence modules via graphcodes} 

\titlerunning{Modules via graphcodes} 

\author{Michael Kerber}{Institute of Geometry, Graz University of Technology, Austria}{kerber@tugraz.at}{https://orcid.org/0000-0002-8030-9299}{Austrian Science Fund (FWF) grant P 33765-N} 

\author{Florian Russold}{Institute of Geometry, Graz University of Technology, Austria}{russold@tugraz.at}{https://orcid.org/0009-0003-2978-0477}{Austrian Science Fund (FWF) grant P 33765-N and W1230}

\authorrunning{Kerber, Russold} 

\Copyright{Michael Kerber, Florian Russold} 


\relatedversion{} 

\acknowledgements{}

\nolinenumbers 

\newcommand{\ignore}[1]{}

\usepackage{tikz-cd}
\usetikzlibrary{arrows}
\usepackage[ruled,vlined, linesnumbered]{algorithm2e}
\usepackage{blkarray}
\usepackage{mathtools}
\usepackage{centernot}

\DeclareMathSymbol{\minus}{\mathbin}{AMSa}{"39}

\newcommand{\expand}{\mathrm{Exp}}
\newcommand{\compress}{\mathrm{Com}}
\newcommand{\plus}{\raisebox{0\height}{\scalebox{.6}{+}}}

\newcommand{\epm}{P}
\newcommand{\pres}{\mathcal{P}}

\begin{document}

\maketitle

\begin{abstract}
Graphcodes were recently introduced as a technique to
employ two-parameter persistence modules in machine learning tasks
(Kerber and Russold, NeurIPS 2024).
We show in this work that a compressed version of graphcodes
yields a description of a two-parameter module that is equivalent
to a presentation of the module.
This alternative representation as a graph allows for a simple
translation between combinatorics and algebra: connected components
of the graphcode correspond to summands of the module and isolated paths correspond to intervals.
We demonstrate that graphcodes are useful in practice by speeding-up
the task of decomposing a module into indecomposable summands.
Also, the graphcode viewpoint allows to devise a simple algorithm
to decide whether a persistence module is interval-decomposable
in $O(n^4)$ time,
which improves on the previous bound of $O(n^{2\omega+1})$.
\end{abstract}

\section{Introduction}
Multi-parameter persistent homology is a rapidly evolving research branch
in topological data analysis. Recent years have witnessed substantial
advances in algorithms, especially for the case of two parameters,
which enable the use of multi-parameter invariants in applications~\cite{BotnanLesnick}.

\emph{Graphcodes}, recently introduced by Kerber and Russold~\cite{russold2024graphcode},
provide a complete summary of a two-parameter persistence module
over $\mathbb{Z}_2$.
They consider such a module as a stack of one-parameter modules
connected by maps between consecutive modules in the stack.
Describing every one-parameter module as a persistent barcode
and every connecting map as a bipartite graph yields a vertex-labeled graph
that can be input to a graph neural network architecture and used in machine
learning pipelines, for instance for image or graph classification tasks.

\subparagraph{Our results.} 
We demonstrate that considering two-parameter persistence
modules as a stack of barcodes with connecting maps,
or algebraically speaking as a persistence module
of persistence modules, 
provides a representation of persistence modules alternative
to a (minimal) presentation that nicely reveals certain aspects
of the module and can be leveraged in computational tasks.
In detail, our contributions are as follows:
\begin{itemize}
\item We extend the graphcode in~\cite{russold2024graphcode} to \emph{compressed graphcodes}
which are significantly smaller in practice and can be computed
even more efficiently than the uncompressed counterpart.
\item We show that graph-theoretic concepts of the graphcode translate
into interesting properties of the underlying persistence module:
connected components of the graphcode correspond to summands of the
persistence module (which are not necessarily indecomposable), 
and components that form a path correspond
to intervals in the module.
\item We give efficient algorithms to compute a graphcode from a minimal presentation and vice versa. Both directions are based on matrix reduction and have cubic worst-case complexity.
\item We demonstrate that compressed graphcodes can be computed fast
in practice and yield representations whose size is comparable to that of minimal presentations.
\item We show that preprocessing a persistence module through graphcodes
to split it into summands helps for the task of computing an indecomposable
decomposition of a persistence module,
in combination with a very recent decomposition algorithm~\cite{aida}. 
\item We present a simple $O(n^4)$ algorithm to decide whether
  a distinctly-graded two-parameter presentation of size $n$ 
  is interval-decomposable.
  This improves the best-known previous
  bound of $O(n^{2\omega+1})$ for this problem by Dey and Xin~\cite{DeyXin2022}
  in the distinctly-graded case (their algorithm is phrased for
  simplicial input but works equally for presentations).\footnote{The aforementioned paper~\cite{aida} outlines a different $O(n^3)$-algorithm for the same problem.}
  Our algorithm is based on standard matrix reduction and can be seen as a
  generalization of the $1$-parameter persistence algorithm of general persistence
  modules as described in~\cite{cdn-persistent,dey_et_al:LIPIcs.SoCG.2024.51,jnt-space}.
\end{itemize}

\subparagraph{Related work.}
The idea of examining a two-parameter persistence module through $1$-parameter
restrictions has appeared in different forms, for instance when computing
the matching distance between persistence modules~\cite{bcfg-new,Biasotti2008,bk-asymptotic,klo-exact,kn-efficient}, when visualizing two-parameter modules~\cite{lesnick2015interactive}, or for sliced Wasserstein
kernels~\cite{DBLP:conf/icml/CarriereCO17,DBLP:conf/scalespace/RabinPDB11}.
A crucial difference between graphcodes and these approaches is
that graphcodes only consider $1$-parameter restrictions in one direction,
but also capture the relation between consecutive levels. In that sense,
while a graphcode is not an invariant of the persistence module, it captures
the complete information of the persistence module, while the previous
approaches focus on the rank invariant and hence are incomplete.

Persistent Vineyards~\cite{cem-vines} are another approach that considers
a stack of persistence diagrams and tracks the continuous trajectories of the dots
along the diagrams. While similar in spirit,
vineyards work in a different setting where the simplicial complex is fixed
and only its filtration order changes over time. In contrast, in a bifiltered
simplicial complex, the filtration increases from level to level.
This leads to features impossible in vineyards, for instance dots in a
persistence diagram can branch out, or die out away from the diagonal.

Decomposing a persistence module into indecomposable summands is a natural
way to understand the structure of a module, generalizing the
barcode/persistence diagram
of $1$-parameter persistence. The decomposition is essentially unique
due to the Krull-Remak-Schmidt theorem, but the indecomposable elements
can be arbitrarily complicated. Nevertheless, the decomposition for a concrete
module can be computed in polynomial time using the more general
meataxe algorithm~\cite{Parker84}. Dey and Xin~\cite{DeyXin2022} give
a faster algorithm tailored for the case of persistence modules
based on the assumption that the presentation is distinctly graded.
Dey, Jendrysiak, and Kerber~\cite{aida} very recently lifted this assumption
and provide an efficient implementation of their algorithm in the \textsc{aida} library.
While graphcodes do not yield the full decomposition, they still partially
reveal the structure of the module and, for instance,
facilitate a full decomposition as we demonstrate.

Intervals are the simplest indecomposable summands, generalizing the
birth-death pairs in the $1$-parameter case. The special case
of modules that fully decompose into intervals received special attention
in the literature. For instance,
such modules allow for a complete invariant~\cite{ckm-discriminating},
for more efficient algorithmic approaches~\cite{dx-computing}
and for better approximation guarantees~\cite{lcb-efficient}.
Also, persistence modules have many intervals,
and some of them can be extracted quickly in the case of $0$-dimensional
homology~\cite{ak-decomposition}~-- our result can be seen as
a generalization for arbitrary modules, as we can detect intervals
by just identifying paths in the graphcode.

While recent work proves that fully interval-decomposable modules are rare
in a probabilistic sense~\cite{aks-probabilistic}, it is still useful
to determine whether a module has this property. The approach by Dey, Kim,
and Memoli~\cite{DBLP:journals/dcg/DeyKM24} is based on a predicate that
decides quickly whether a given interval is a summand of the module,
and iterates over a set of possible interval shapes; see also~\cite{ABENY}.
Our algorithm, in contrast, avoids such an iteration and rather constructs
the intervals in a greedy fashion and gives up if it detects an obstruction.

\subparagraph{Outline.}
We review basic definitions in Section~\ref{sec:basic_notions}.
We define (compressed) graphcodes and show their major properties in
Section~\ref{sec:graphcodes}. We present the algorithms to convert
between minimal presentations and compressed graphcodes in Section~\ref{sec:computation}.
We evaluate these algorithms experimentally in Section~\ref{sec:experiments}
and show their usefulness for module decompositions.
We describe the algorithm to decide on interval-decomposability
in Section~\ref{sec:interval_decomposability}.
We conclude in Section~\ref{sec:conclusion}. 

\section{Persistence modules}
\label{sec:basic_notions}

We denote by $G(m,n)$ the poset $\{1,\ldots,m\}\times \{1,\ldots,n\}$ with the product order. A \emph{persistence module} is a functor $M\colon G(m,n)\rightarrow \mathbf{Vec}$ where we assume all vector spaces to be finite dimensional over $\mathbb{Z}_2$. In other words, a persistence module is a digram of vector spaces and linear maps over $G(m,n)$ as depicted in \eqref{eq:persmod_slices} on the left. We denote by $\mathbf{Vec}^{G(m,n)}$ the category of persistence modules where the morphisms are natural transformations.
\begin{equation} \label{eq:persmod_slices}
\begin{tikzcd}[column sep=large,row sep=large]
M_{(1,n)} \arrow[r,"M_{(1,n)}^{(2,n)}"] & M_{(2,n)} \arrow[r,"M_{(2,n)}^{(3,n)}"] & \cdots \arrow[r,"M_{(m\minus 1,n)}^{(m,n)}"] &[5pt] M_{(m,n)} &[-10pt] \colon &[-10pt] M_{(\minus,n)} \\
\vdots \arrow[u,"M_{(1,n\minus 1)}^{(1,n)}"] & \vdots \arrow[u,swap,"M_{(2,n\minus 1)}^{(2,n)}"] & & \vdots \arrow[u,swap,"M_{(m,n\minus 1)}^{(m,n)}"] & & \vdots \arrow[u,swap,"M_{(\minus,n\minus 1)}^{(\minus,n)}"] \\
M_{(1,2)} \arrow[r,"M_{(1,2)}^{(2,2)}"] \arrow[u,"M_{(1,2)}^{(1,3)}"] & M_{(2,2)} \arrow[r,"M_{(2,2)}^{(3,2)}"] \arrow[u,swap,"M_{(2,2)}^{(2,3)}"] & \cdots \arrow[r,"M_{(m\minus 1,2)}^{(m,2)}"] & M_{(m,2)} \arrow[u,swap,"M_{(m,2)}^{(m,3)}"] & \colon & M_{(\minus,2)}  \arrow[u,swap,"M_{(\minus,2)}^{(\minus,3)}"]  \\
M_{(1,1)} \arrow[r,"M_{(1,1)}^{(2,1)}"] \arrow[u,"M_{(1,1)}^{(1,2)}"] & M_{(2,1)} \arrow[r,"M_{(2,1)}^{(3,1)}"] \arrow[u,swap,"M_{(2,1)}^{(2,2)}"] & \cdots \arrow[r,"M_{(m\minus 1,1)}^{(m,1)}"] & M_{(m,1)} \arrow[u,swap,"M_{(m,1)}^{(m,2)}"] & \colon & M_{(\minus,1)}  \arrow[u,swap,"M_{(\minus,1)}^{(\minus,2)}"] 
\end{tikzcd}
\end{equation}
Throughout this work, we will adopt an alternative perspective on two-parameter persistence modules as persistence modules of one-parameter persistence modules:
If we fix the second coordinate of $M_{(k,l)}$ in \eqref{eq:persmod_slices} to $l=i$, we obtain the horizontal slice $M_{(\minus,i)}$ which is a one-parameter persistence module $M_{(\minus,i)}\colon G(m,1)\rightarrow \mathbf{Vec}$. For any two consecutive slices $M_{(\minus,i)}$ and $M_{(\minus,i\plus 1)}$, the family of vertical linear maps $M_{(\minus,i)}^{(\minus,i\plus 1)}$ connecting the slices form a morphism of one-parameter persistence modules $M_{(\minus,i)}^{(\minus,i\plus 1)}\colon M_{(\minus,i)}\rightarrow M_{(\minus,i\plus 1)}$. Therefore, the two-parameter persistence module $M$ is a persistence module of one-parameter persistence modules $M\colon G(1,n)\rightarrow \mathbf{Vec}^{G(m,1)}$ depicted in \eqref{eq:persmod_slices} on the right. Conversely, a persistence module of one-parameter persistence modules $M\colon G(1,n)\rightarrow \mathbf{Vec}^{G(m,1)}$ determines a two-parameter persistence module. We will refer to this representation as the \emph{slicewise perspective} in the text.

\ignore{
In applications persistence modules typically arise from the homology of (bi)filtered simplicial complexes. A \emph{bifiltered simplicial complex} is a functor $K\colon G(m,n)\rightarrow \mathbf{SCpx}$ such that $K(x\leq y)$ is an inclusion for all $x\leq y\in G(m,n)$. If $K$ is a bifiltration, then $H_\ell(K)\colon G(m,n)\rightarrow\mathbf{Vec}$ is a persistence module. We note that we do not use real valued indexing, i.e.\ functors $\mathbb{R}^2\rightarrow \mathbf{Vec}$ as it is not necessary in our context but we can always extend a persistence module $M\colon G(m,n)\rightarrow\mathbf{Vec}$ to a persistence module $M'\colon \mathbb{R}^2\rightarrow \mathbf{Vec}$ along an embedding $\iota\colon G(m,n)\xhookrightarrow{} \mathbb{R}^2$ assigning the points in $G(m,n)$ real coordinates. \textcolor{red}{Maybe mention distinctly graded and one-critical.}
}

We define the \emph{support} of $M$ as $\text{supp}(M)\coloneqq\{x\in G(m,n)\vert M(x)\neq 0\}$. A persistence module is called \emph{thin} if $\text{dim }M(x)\leq 1\hspace{2pt}\forall x\in G(m,n)$. An indecomposable thin persistence module whose support is a convex connected subset of $G(m,n)$ is called an \emph{interval module}. By the theorems of Gabriel and Krull-Remak-Schmidt, every one-parameter persistence module $M\colon G(m,1)\rightarrow\mathbf{Vec}$ is isomorphic to a sum of interval modules $\bigoplus_{k=1}^s I_k$. 
Crucially for us, morphisms between one-parameter interval modules are very simple. Consider two interval modules $I,J\colon G(m,1)\rightarrow \mathbf{Vec}$ such that $I=[b_1,d_1)$ and $J=[b_2,d_2)$ as depicted in \eqref{eq:interval_module_entanglement}. A morphism from $I$ to $J$ is a sequence of vertical maps as in \eqref{eq:interval_module_entanglement} such that all squares commute.
\begin{equation} \label{eq:interval_module_entanglement}
\begin{tikzcd}
&[-2pt] \color{red} b_2 &[-2pt] &[-2pt] \color{red} b_1 &[-2pt] &[-2pt] &[-2pt] &[-2pt] \color{red}d_2 &[-3pt] \color{red} d_1\phantom{0} \\[-17pt]
0 \arrow[r] & \mathbb{Z}_2 \arrow[r,"1"] & \mathbb{Z}_2 \arrow[r,"1"] & \mathbb{Z}_2 \arrow[r,"1"] & \mathbb{Z}_2 \arrow[r,"1"] & \mathbb{Z}_2 \arrow[r,"1"] & \mathbb{Z}_2 \arrow[r] & 0 \arrow[r]  & 0 &[-25pt] \colon &[-25pt] J \\
0 \arrow[r] \arrow[u] & 0 \arrow[r] \arrow[u] & 0 \arrow[r] \arrow[u] & \mathbb{Z}_2 \arrow[r,"1"] \arrow[u,"1"] & \mathbb{Z}_2 \arrow[r,"1"] \arrow[u,"1"] & \mathbb{Z}_2 \arrow[r,"1"] \arrow[u,"1"] & \mathbb{Z}_2 \arrow[r,"1"] \arrow[u,"1"] & \mathbb{Z}_2 \arrow[r] \arrow[u] & 0 \arrow[u] & \colon & I \arrow[u,swap,"1"] \\[-17pt]
\end{tikzcd}
\end{equation}
This condition implies that a morphism that is non-zero at some point in $I\cap J$ has to be non-zero at every point in $I\cap J$. It also implies that the interval $J$ cannot start or end later than $I$. In order words,
\begin{equation}
\text{Hom}(I,J)\cong\begin{cases} \mathbb{Z}_2 & \colon \text{ if}\hspace{5pt}b_2\leq b_1<d_2\leq d_1 \\ 0 & \colon \text{ otherwise} \end{cases}
\end{equation}
which means that there is either no or a unique non-zero morphism between two-interval modules. We call two intervals $I$ and $J$, satisfying $b_2\leq b_1<d_2\leq d_1$, to be \emph{entangled}, which we denote by $J\lhd I$.

\section{Graphcodes}
\label{sec:graphcodes}

A graphcode is a combinatorial description of a two-parameter persistence module as a graph.
They are most naturally introduced by considering a special case of persistence modules first:
Consider a sequence of one-parameter persistence modules $M_{(\minus,1)},\ldots,M_{(\minus,n)}$
where every module is a direct sum of interval modules,
i.e.\ $M_{(\minus,i)}= \bigoplus_{k=1}^{s_i} I_k^i$ with every $I_k^i$ a one-parameter interval.
Let $\eta^i$ denote a morphism from $M_{(\minus,i)}$ to $M_{(\minus,i\plus 1)}$. Then, the diagram
\begin{equation} \label{eq:interval_module_module}
\begin{tikzcd}[column sep=large]
\bigoplus_{k=1}^{s_1} I^1_k \arrow[r,"\eta^1"] & \bigoplus_{k=1}^{s_2} I^2_k \arrow[r,"\eta^2"] & \cdots \arrow[r,"\eta^{n\minus 1}"] & \bigoplus_{k=1}^{s_n} I^n_k 
\end{tikzcd}
\end{equation}
yields a two-parameter persistence module $M$ (following the slicewise perspective from Section~\ref{sec:basic_notions}).
Moreover, since 
\begin{equation}
\eta^i\in \text{Hom}\left(\bigoplus_{k=1}^{s_i} I_k^i\hspace{2pt},\hspace{2pt}\bigoplus_{l=1}^{s_{i\plus 1}} I_l^{i\plus 1}\right)\cong \bigoplus_{l=1}^{s_{i\plus 1}}\bigoplus_{k=1}^{s_i}\text{Hom}\left(I_k^i,I_l^{i \plus1}\right)
\end{equation}
is a morphism between direct sums of intervals, it is completely described by the morphisms $\eta^i_{lk}\colon I^i_k\rightarrow I^{i\plus 1}_l$ between individual summands. As discussed, $\text{Hom}\left(I_k^i,I_l^{i \plus1}\right)$ either contains no or a unique non-zero morphism. This allows us to describe $\eta^i$ by a matrix 
\begin{equation}
\eta^i=
\begin{blockarray}{cccc}
& I_1^i & \cdots & I_{s_i}^i \\[3pt]
\begin{block}{c(ccc)}
I_1^{i\plus 1} & \eta^i_{11} & \cdots & \eta^i_{1s_i} \\[2pt] 
\vdots & \vdots & \ddots & \vdots \\[2pt]
I_{s_{i\plus 1}}^{i\plus 1} & \eta^i_{s_{i\plus 1}1} & \cdots & \eta^i_{s_{i\plus 1}s_i} \\
\end{block}
\end{blockarray}  
\end{equation}
with coefficients $\eta^i_{lk}\in\mathbb{Z}_2$ where $\eta^i_{lk}=1$ denotes the unique non-zero morphism and $\eta^i_{lk}=0$ the zero morphism. Therefore, $M$ can be described as a sequence of matrices $\eta^1,\ldots,\eta^{n\minus 1}$ whose columns and rows are labeled with intervals, as depicted at the top of Figure~\ref{fig:example_graphcode}. It is easy to see how this description induces a two-parameter persistence module: At a point $(x,y)\in G(m,n)$ the vector space $M_{(x,y)}$ is generated by all intervals $I^y_k$ such that $x\in I_k^y$. The horizontal map $M_{(x,y)}^{(x\plus 1,y)}$ sends a basis element corresponding to $I^y_k$ to itself if $x+1\in I^y_k$ and to zero otherwise. The vertical map $M_{(x,y)}^{(x,y\plus 1)}$ is defined by the matrix $\eta^y$ restricted to all intervals $I^{y}_k$ and $I^{y\plus 1}_l$ that contain $x$. Note that, except for the interval labels, this description of $M$ by $\eta^1,\ldots,\eta^{n\minus 1}$ is completely analogous to the description of a one-parameter persistence module by a sequence of matrices. 

This description of $M$ suggests the following combinatorial representation depicted in Figure \ref{fig:example_graphcode}: The intervals $I_1^i,\ldots,I_{s_i}^i$ in each $M_{(\minus,i)}$ can be summarized as a persistence diagram. In this way we obtain a stack of persistence diagrams where the points in the $i$-th diagram sit at height $i$. The matrices $\eta^i$ can be viewed as adjacency matrices of bipartite graphs with vertices the persistence diagrams of $M_{(\minus,i)}$ and $M_{(\minus,i\plus 1)}$, where a non-zero entry $\eta^i_{lk}$ implies an edge from $I_k^i$ to $I_l^{i\plus 1}$. These observations motivate the following definition.

\begin{figure}
    \centering
    \includegraphics[width=0.8\linewidth]{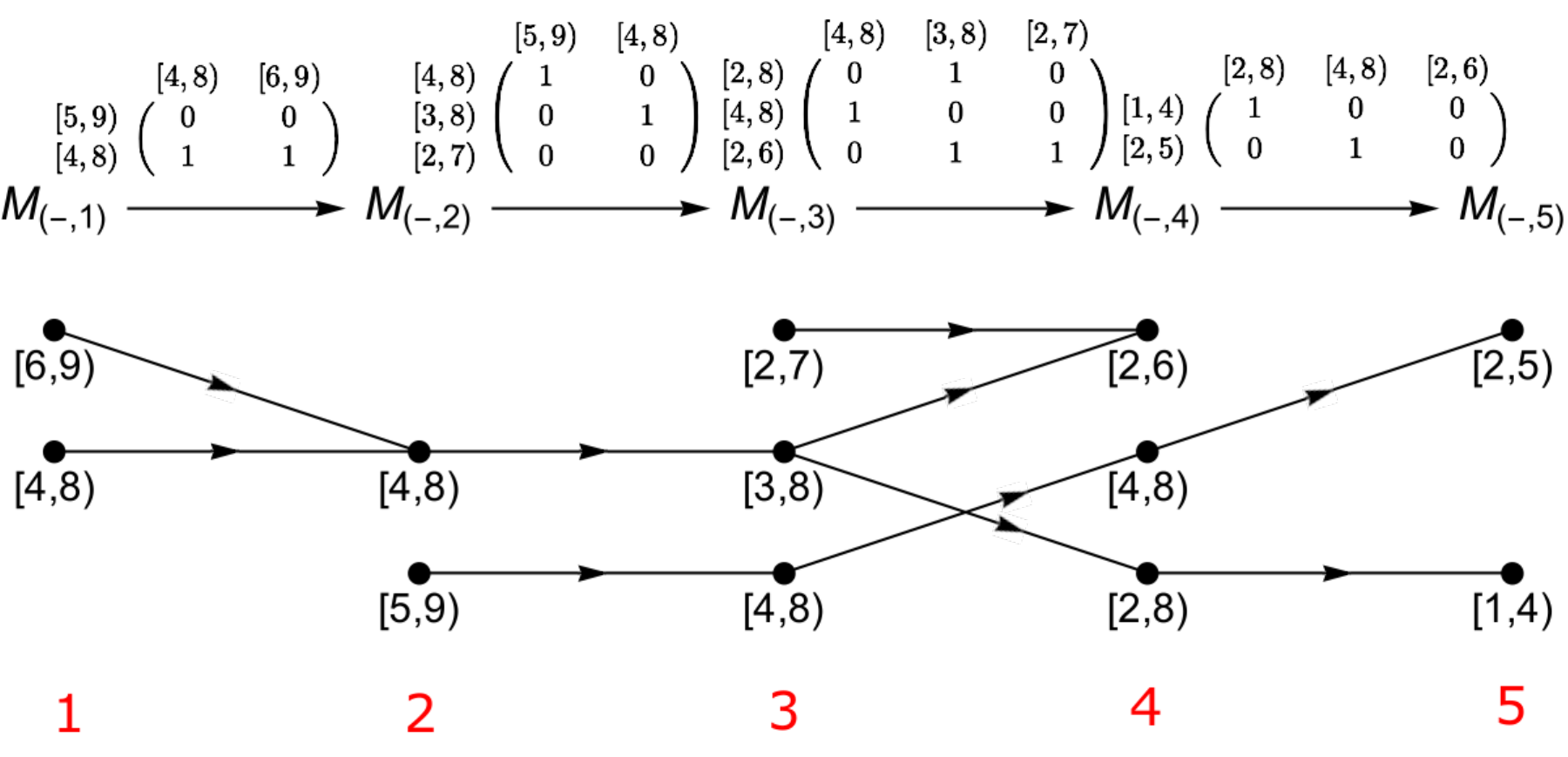}
    \caption{Construction of graphcode from matrix representations.}
    \label{fig:example_graphcode}
\end{figure}

\begin{definition}[Graphcode]
  \label{def:graphcode}
A directed graph $\mathcal{G}=(V,E,c)$ with vertex labels 
\begin{equation*}
c\colon V\rightarrow \{1,\ldots,m\}\times\{1,\ldots,m+1\}\times\{1,\ldots,n\}
\end{equation*}
is called a \emph{graphcode} if:
\begin{enumerate}
    \item for all $v\in V$ with $c(v)=(b,d,h)$, we have $b<d$
    \item for all $(v,w)\in E$ with $c(v)=(b_1,d_1,h_1)$ and $c(w)=(b_2,d_2,h_2)$, \\ we have $h_2=h_1+1$ and $[b_2,d_2)\lhd [b_1,d_1)$.
\end{enumerate}
\end{definition}
In words, for $c(v) = \big(b(v), d(v), h(v)\big)$, we call the coordinates of $c(v)$ the birth, death, and height of
$v$, respectively. Then, the first property implies that $v$ can be interpreted as an interval with lifetime $[b, d)$, belonging to a persistent barcode at height $h$. The second property implies that
all edges in the graphcode are between vertices of consecutive heights and entangled intervals. In contrast to \cite{russold2024graphcode}, we use directed edges for the definition of graphcodes to simplify certain definitions and arguments.

The slicewise persistence module $M$ in \eqref{eq:interval_module_module} induces a graphcode $\mathcal{GC}(M)\coloneqq(V,E,c)$, according to Definition \ref{def:graphcode}, in the following way: 
\begin{align*}
& V\coloneqq\bigcup_{i=1}^n \{I^i_1,\ldots, I^i_{s_i}\} \\
& E\coloneqq\{(I^i_k,I^{i\plus 1}_l)\vert \eta^i_{lk}=1\} \\[5pt]
& c(I^i_k)\coloneqq \big(b,d,i\big) \hspace{5pt} \text{for} \hspace{5pt} I^i_k=[b,d) .
\end{align*} 

Conversely, a graphcode $\mathcal{G}=(V,E,c)$ induces a persistence module of one-parameter persistence modules $PM(\mathcal{G})\colon G(1,n)\rightarrow \mathbf{Vec}^{G(m,1)}$ and, thus, a two-parameter persistence module, in the following way: For every $1\leq i\leq n$, define $PM(\mathcal{G})_{(\minus,i)}\coloneqq\underset{v\colon h(v)=i}{\bigoplus} [b(v),d(v))$.  For all $1\leq i<n$, define $\eta^i$ as the matrix whose columns are labeled by the intervals $[b(v),d(v))$ with $h(v)=i$, whose rows are labeled by the intervals $[b(w),d(w))$ with $h(w)=i+1$ and whose entries $\eta^i_{wv}$ are $1$ if $(v,w)\in E$ and $0$ otherwise. By the first property of a graphcode, $[b(v),d(v))$ defines a valid one-parameter interval module for all $v\in V$. By the second property of a graphcode, for an edge $(v,w)$, we have $[b(w),d(w))\lhd [b(v),d(v))$. Hence, there exists a non-zero morphism of interval modules $[b(v),d(v))\xrightarrow{1} [b(w),d(w))$ and $\eta^i$ induces a well-defined morphism $PM(\mathcal{G})_{(\minus,i)}^{(\minus,i\plus 1)}\colon PM(\mathcal{G})_{(\minus,i)}\rightarrow PM(\mathcal{G})_{(\minus,i\plus 1)}$. This construction of persistence modules from graphcodes has the following nice properties which follow directly from the definition.

\begin{proposition} \label{prop:persmod_from_graphcode_union}
If $\mathcal{G}_1$ and $\mathcal{G}_2$ are graphcodes that are disjoint as graphs, then $PM(\mathcal{G}_1\sqcup \mathcal{G}_2)=PM(\mathcal{G}_1)\oplus PM(\mathcal{G}_2)$.
\end{proposition}

\noindent
As an example, the persistence module induced by the graphcode shown in Figure \ref{fig:example_graphcode}, is the direct sum of the persistence modules induced by its two connected components. 

\begin{figure}
    \centering
    \includegraphics[width=0.95\linewidth]{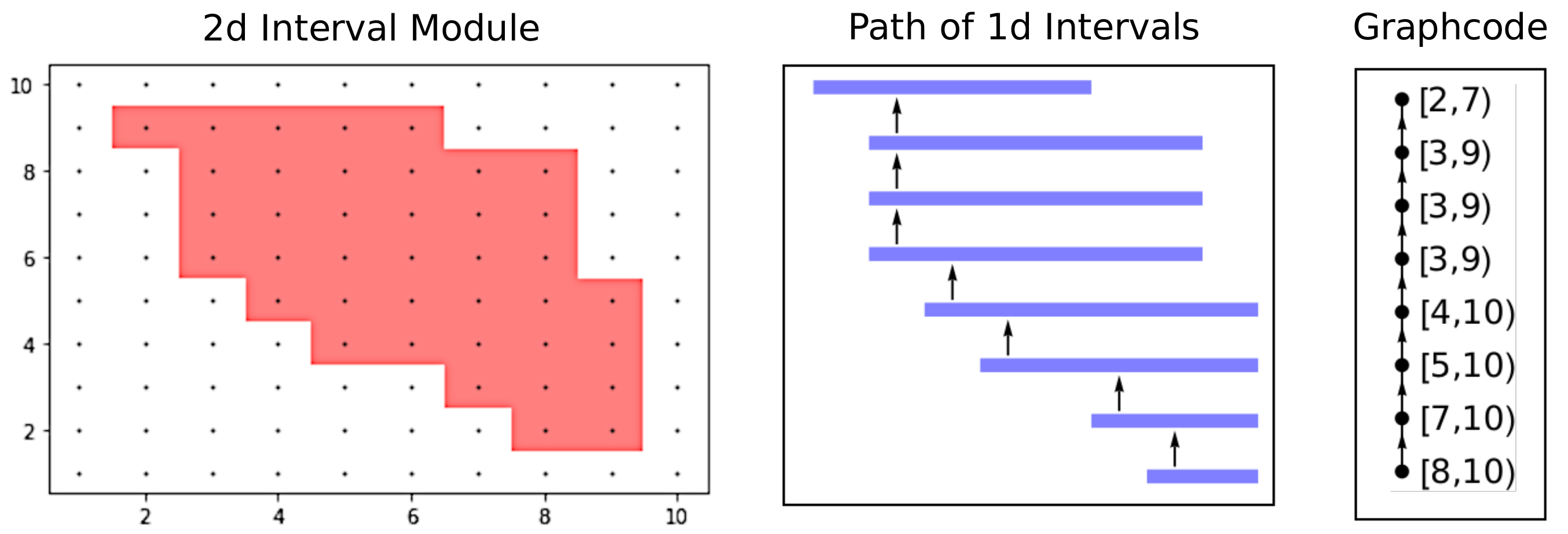}
    \caption{Graphcode of interval module supported on the red shaded area.}
    \label{fig:graphcode_interval}
\end{figure}

\begin{proposition} \label{prop:persmod_from_path}
If $\mathcal{G}$ is a graphcode whose underlying graph is a directed path, then $PM(\mathcal{G})$ is an interval module.
\end{proposition}

\noindent
Figure \ref{fig:graphcode_interval} shows an example of an interval module induced by a path.

\subparagraph{Graphcodes from arbitrary persistence modules.}  
We now consider an arbitrary persistence module $M\colon G(m,n)\rightarrow\mathbf{Vec}$ which we view as a persistence module of one-parameter persistence modules $M\colon G(1,n)\rightarrow\mathbf{Vec}^{G(m,1)}$. We know that every one-parameter persistence module $M_{(\minus,i)}$ is isomorphic to a direct sum of interval modules. Let $\mu_i\colon M_{(\minus,i)}\xrightarrow{\cong} \bigoplus_{k=1}^{s_i}I^i_k$ be an isomorphism. We call such an isomorphism $\mu_i$ a \emph{barcode basis}. Moreover, we call a collection of barcode bases $\mu=(\mu_1,\ldots,\mu_n)$ a barcode basis for $M$. Note that the isomorphism $\mu_i$ and, thus, the barcode basis is not unique. A barcode basis $\mu$ induces an isomorphism $\mu\colon M\xrightarrow{\cong} M_\mu$ to a persistence module of one-parameter persistence modules $M_\mu$ where each one-parameter slice $(M_\mu)_{(\minus,i)}$ is a direct sum of intervals:  
\begin{equation} \label{eq:equivalence}
\begin{tikzcd}[column sep=huge]
M_{(\minus,1)} \arrow[r,"M_{(\minus,1)}^{(\minus,2)}"] \arrow[d,"\cong","\mu_1"'] & M_{(\minus,2)} \arrow[r,"M_{(\minus,2)}^{(\minus,3)}"] \arrow[d,"\cong","\mu_2"'] & \cdots \arrow[r,"M_{(\minus,n\minus 1)}^{(\minus,n)}"] & M_{(\minus,n)} \arrow[d,swap,"\cong"',"\mu_n"] &[-30pt] \colon &[-50pt] M \arrow[d,"\mu"',"\cong"] \\
\bigoplus_{k=1}^{s_1} I^1_k \arrow[r,"\eta^1"] & \bigoplus_{k=1}^{s_2} I^2_k \arrow[r,"\eta^2"] & \cdots \arrow[r,"\eta^{n\minus 1}"] & \bigoplus_{k=1}^{s_n} I^n_k & \colon & M_\mu
\end{tikzcd}
\end{equation}
We now define the graphcode of the initial persistence module $M$ with respect to the barcode basis $\mu$ as $\mathcal{GC}_\mu(M)\coloneqq \mathcal{GC}(M_\mu)$. Obviously the persistence module $PM\big(\mathcal{GC}(M_\mu)\big)$, induced by the graphcode $\mathcal{GC}(M_\mu)$, is equal to $M_\mu$ up to permutation of the intervals in every $M_{(\minus,i)}$. Since the initial persistence module $M$ is isomorphic to $M_\mu$ we obtain the following proposition.

\begin{proposition} \label{prop:persmod_reconstruction}
Let $M$ be a persistence module and $\mu$ a barcode basis of $M$, then $PM\big(\mathcal{GC}_\mu(M)\big)\cong M$.
\end{proposition}

\noindent
Proposition \ref{prop:persmod_reconstruction} implies that the graphcode of $M$ with respect to $\mu$ is a complete summary of $M$ as we can reconstruct $M$ from $\mathcal{GC}_\mu(M)$ up to isomorphism. Since the graphcode depends on the choice of barcode basis $\mu$, the graphcode is not an invariant of $M$. However, the graphcode representation allows us to infer algebraic properties of $M$ from combinatorial properties of the graph $\mathcal{GC}_\mu(M)$. Since every graph is the disjoint union of its connected components, Proposition \ref{prop:persmod_from_graphcode_union} and \ref{prop:persmod_reconstruction} imply that the connected components of $\mathcal{GC}_\mu(M)$ correspond to summands of $M$. We now show that the converse of Proposition \ref{prop:persmod_from_graphcode_union} and \ref{prop:persmod_from_path} is also true. This gives us a characterization of interval modules via graphcodes and will ultimately lead to a characterization of interval-decomposable modules discussed in Section \ref{sec:interval_decomposability}.

\begin{proposition} \label{prop:graphcode_from_sum}
Let $M_1$ and $M_2$ be persistence modules and $\mu_1$ and $\mu_2$ barcode bases, then $\mathcal{GC}_{\mu_1\oplus\mu_2}(M_1\oplus M_2)=\mathcal{GC}_{\mu_1}(M_1)\sqcup\mathcal{GC}_{\mu_2}(M_2)$.
\end{proposition}

\begin{proof}
Since $\mu_1$ and $\mu_2$ are barcode bases of $M_1$ and $M_2$, respectively, $\mu_1\oplus\mu_2$ is a barcode basis for $M_1\oplus M_2$. We obtain isomorphisms $\mu_1\colon M_1\xrightarrow{\cong}M_{\mu_1}$ , $\mu_2\colon M_2\xrightarrow{\cong}M_{\mu_2}$ and $\mu_1\oplus\mu_2\colon M_1\oplus M_2\xrightarrow{\cong} M_{\mu_1}\oplus M_{\mu_2}$. We can partition the vertices in $\mathcal{GC}(M_{\mu_1}\oplus M_{\mu_2})$ into vertices in $\mathcal{GC}(M_{\mu_1})$ and vertices in $\mathcal{GC}(M_{\mu_2})$. By construction as a direct sum, if there is a non-zero map $I^i_k\rightarrow I^{i\plus 1}_l$ in the persistence module of one-parameter persistence modules $M_{\mu_1}\oplus M_{\mu_2}$, then either both $I^i_k$ and $I^{i\plus 1}_{l}$ belong to $M_{\mu_1}$ or both belong to $M_{\mu_2}$. Hence, there are no edges between vertices in $\mathcal{GC}(M_{\mu_1})$ and vertices in $\mathcal{GC}(M_{\mu_2})$ and we have $\mathcal{GC}(M_{\mu_1}\oplus M_{\mu_2})=\mathcal{GC}(M_{\mu_1})\sqcup \mathcal{GC}(M_{\mu_2})$. By definition we obtain $\mathcal{GC}_{\mu_1\oplus\mu_2}(M_1\oplus M_2)=\mathcal{GC}_{\mu_1}(M_1)\sqcup\mathcal{GC}_{\mu_2}(M_2)$. 
\end{proof}

\begin{corollary} \label{cor:sum_to_union}
If $M$ is a persistence module such that $M\cong M_1\oplus M_2$, then there exist barcode bases $\mu_1$, $\mu_2$ and $\mu$ such that $\mathcal{GC}_\mu(M)=\mathcal{GC}_{\mu_1}(M_1)\sqcup \mathcal{GC}_{\mu_2}(M_2)$.
\end{corollary}

\begin{proof}
Let $\nu\colon M\xrightarrow{\cong}M_1\oplus M_2$ be an isomorphism. Since $\mu_1$ and $\mu_2$ are barcode bases of $M_1$ and $M_2$, respectively, $\mu\coloneqq (\mu_1\oplus\mu_2)\circ\nu$ is a barcode basis for $M$. Thus, we have an isomorphism $\mu\colon M\xrightarrow{\cong} M_{\mu_1}\oplus M_{\mu_2}$ and $\mathcal{GC}_\mu(M)=\mathcal{GC}(M_{\mu_1}\oplus M_{\mu_2})=\mathcal{GC}_{\mu_1\oplus\mu_2}(M_1\oplus M_2)$. By Proposition \ref{prop:graphcode_from_sum}, $\mathcal{GC}_\mu(M)=\mathcal{GC}_{\mu_1}(M_1)\sqcup\mathcal{GC}_{\mu_2}(M_2)$.
\end{proof}

\noindent
If $M$ is an interval module, then, when viewed as a persistence module of one-parameter persistence modules, every slice $M_{(\minus,i)}$ is either zero or a single one-parameter interval (see Figure \ref{fig:graphcode_interval}). This is a direct consequence of the definition of an interval module as an indecomposable thin persistence module supported on a convex connected subset of $G(m,n)$. Since there is at most one non-zero morphism between two one-parameter interval modules, $\text{id}=(\text{id},\ldots,\text{id})$ is the only barcode basis for an interval module. 

\begin{proposition} \label{prop:graphcode_from_interval}
For an interval module $M$, $\mathcal{GC}_\text{id}(M)$ is a directed path.
\end{proposition}

\begin{proof}
By construction of $\mathcal{GC}_\text{id}(M)=\mathcal{GC}(M)$, the vertices at every height of $\mathcal{GC}(M)$ correspond to the intervals in $M_{(\minus,i)}$. Since $M$ is an interval module, there is at most one vertex at every height. Since $M$ is indecomposable all these vertices have to be connected. Otherwise Proposition \ref{prop:persmod_from_graphcode_union} and \ref{prop:persmod_reconstruction} would imply that $M$ is decomposable. Since there are only edges between vertices at consecutive height levels in a graphcode, $\mathcal{GC}(M)$ is a directed path.
\end{proof}

\noindent
\textbf{Compressed graphcodes.}
A disadvantage of graphcodes is that every bar of every barcode is represented by a vertex, leading to a high number of vertices and therefore a large graph.
However, in practical situations, the barcodes of consecutive slices often do not differ from each other, and the question is how to compress the information
of a graphcode without losing any information. The following definition prepares our approach by relaxing the condition that graphcodes only connect bars
on consecutive slices.

\begin{definition}[generalized Graphcode]
A directed graph $\mathcal{G}=(V,E,c)$ with vertex labels 
\begin{equation*}
c\colon V\rightarrow \{1,\ldots,m\}\times\{1,\ldots,m+1\}\times\{1,\ldots,n\}
\end{equation*}
is called a \emph{generalized graphcode} if:
\begin{enumerate}
\item for all $v\in V$ with $c(v)=(b,d,h)$, we have $b<d$
\item for all $v\in V$ and all $(v,w),(v,w')\in E$, we have $h(w)=h(w')$
\item for all $(v,w)\in E$ with $c(v)=(b_1,d_1,h_1)$ and $c(w)=(b_2,d_2,h_2)$, \\ we have $[b_2,d_2)\lhd [b_1,d_1)$.
\end{enumerate}
\end{definition}

\noindent
Note that graphcodes are generalized graphcodes where $h(w)=h(v)+1$ for all edges $(v,w)\in V$.
The condition that all out-neighbors
of a vertex are on the same height will be convenient in the next section where we generate a minimal presentation out of a generalized graphcode.

We call a vertex $w$ in a generalized graphcode $\mathcal{G}$ \emph{superfluous} if $w$ has exactly one incoming edge $(v,w)$, $v$ and $w$ have the same
birth and death values (i.e., with $c(v)=(b_1,d_1,h_1)$ and $c(w)=(b_2,d_2,h_2)$, we have $b_1=b_2$ and $d_1=d_2$), $w$ has at least one outgoing edge,
and $v$ has no further outgoing edge. In that case, writing $(w,x_1),\ldots,(w,x_s)$ for the outgoing edges of $w$, we define $\mathcal{G}'$ as the graph
obtained from $\mathcal{G}$
by removing $w$ and all incident edges, and adding the edges $(v,x_1),\ldots,(v,x_s)$ instead. It is simple to verify that $\mathcal{G}'$ is a generalized
graphcode as well. Moreover, the property of being superfluous does not change when eliminating other superfluous vertices in the graphcode.
We call a generalized graphcode $\mathcal{G}'$ a \emph{compression} of a graphcode $\mathcal{G}$ if $\mathcal{G}'$ is obtained from $\mathcal{G}$ by removing
superfluous vertices. We call $\mathcal{G}'$ \emph{fully compressed} if it does not contain any superfluous vertex.
Figure \ref{fig:compressed_graphcode} shows an example.

A generalized graphcode also induces a two-parameter persistence module: For a generalized graphcode $\mathcal{G}$, we define its
\emph{expansion} $\expand(\mathcal{G})$ as follows: if a vertex $v$ with $c(v)=(b,d,h)$ has all its outgoing neighbors $w_1,\ldots,w_s$ at height $h'$ with $h'-h\geq 2$,
remove all edges $(v,w_i)$, introduce vertices $v_{h\plus 1},\ldots,v_{h'\minus 1}$ with $c(v_{h\plus i}):=(b,d,h+i)$ and add edges
$(v,v_{h\plus 1}),(v_{h\plus 1},v_{h\plus 2}),\ldots,(v_{h'\minus 2},v_{h'\minus 1})$ as well as $(v_{h'\minus 1},w_1),\ldots,(v_{h'\minus 1},w_s)$.
Then $\expand(\mathcal{G})$ is a graphcode (in the sense of Definition~\ref{def:graphcode})
and we define $PM(\mathcal{G}):=PM\big(E(\mathcal{G})\big)$.

It is easy to observe that expansion is inverse to compression, in the sense
that if $\compress(H)$ denotes a compression of a graphcode $H$, then
$\expand(\compress(H))=H$. This implies that the persistence modules represented by a graphcode and its compression are equal.

Although the definition of superfluous vertices might appear quite restrictive, a graphcode has typically many such superfluous
vertices as we will demonstrate in Table~\ref{tbl:compressed_table}.
Since the compression process is obviously doable in linear time in the size of a graphcode, we can always assume to work with
fully compressed graphcodes. But even better, the next section will show that a compressed (although not fully compressed)
graphcode can be efficiently computed directly without the detour of computing an uncompressed graphcode first.

\begin{figure}
    \centering
    \includegraphics[width=0.6\linewidth,]{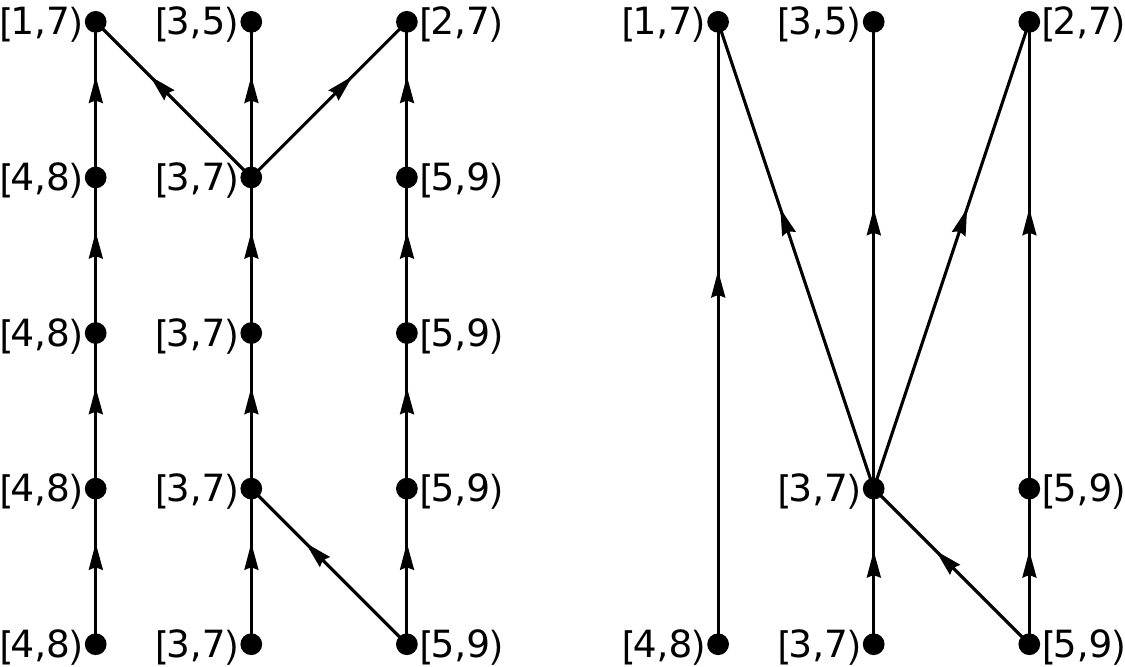}
    \caption{Graphcode (left), compressed graphcode (right).}
    \label{fig:compressed_graphcode}
\end{figure}

\section{From Presentations to Graphcodes and back} \label{sec:computation}

\subparagraph{Presentations.}
Presentations are a standard way to represent two-parameter persistence modules. We quickly recap their definition.
We denote by $\epm[y]$ the \emph{elementary projective module} at $y\in G(m,n)$ which is defined as the interval module with support $\{x\in G(m,n)\vert y\leq x\}$.
A \emph{projective module} is a direct sum of elementary projective modules.
For any persistence module $M$, there exist projective modules $Gen$ and $Rel$, which we call the generators and relations of $M$, and a morphism $p\colon Rel\rightarrow Gen$, which we call a (projective) \emph{presentation} of $M$, such that $M\cong\text{coker }p$.
If $Gen=\bigoplus_{i=1}^m \epm[x_i]$ and $Rel=\bigoplus_{j=1}^n \epm[y_j]$, we call $M$ \emph{finitely presented}, and we obtain $\text{Hom}(Rel,Gen)\cong \bigoplus_{i=1}^m\bigoplus_{j=1}^n \text{Hom}(\epm[y_j],\epm[x_i])$ where each morphism $\epm[y_j]\rightarrow\epm[x_i]$ is either $0$ or $1$.
Therefore, $p$ can be represented as a matrix with coefficients in $\mathbb{Z}_2$ where the columns and rows are labeled with the grades $y_j$ and $x_i$ of the generators and relations in $G(m,n)$.
A presentation of $M$ is \emph{minimal} if $M$ cannot be presented with fewer generators or relations.
The \textsc{mpfree} library efficiently computes a minimal presentation of the homology persistence module of a simplicial bifiltration\footnote{\url{https://bitbucket.org/mkerber/mpfree/src/master/}}; see also~\cite{bll-efficient,lw-computing}.

\subparagraph{The graphcode algorithm.}
Compressed graphcodes can be computed efficiently from a presentation.
The algorithm that we will outline is mostly analogous to the
one described 
in~\cite[App.B]{russold2024graphcode} where a chain complex
is assumed as input instead.
Specifically, we assume
the presentation $\pres$ to be given by a bi-graded matrix
that encodes the generators (as rows) and relations (as columns) of
the presentation.
The graphcode algorithm computes a graphcode $\mathcal{G}$ such
that the persistence modules induced by $\pres$ and $\mathcal{G}$ are isomorphic.

We will refer to the first and second coordinate of the bi-grades
as the \emph{scale} and the \emph{height} of a generator or relation, respectively.
We call a generator or relation \emph{active} at height $i$ if its own height is at most $i$.
We start by reviewing the standard persistence algorithm to compute the barcode
of $M_{(\minus ,i)}$. For that, define the matrix $X_i$ consisting of generators and relations that are active at height $i$,
sorted according to their scale.
The \emph{pivot} of a non-zero column is the index of the lowest non-zero row.
We traverse the matrix from the left to the right and add previously-encountered
columns if their pivot equals the pivot of the current column, until the pivot
of the current column is new or the column is reduced to zero.
When the process finishes, the pivot entries reveal the barcode of $M_{(\minus ,i)}$:
if the pivot of column $j$ is at row index $i$, row $i$ has scale $b$,
and column $j$ has scale $d$, then $[b,d)$ is a bar in the barcode.
The case $b=d$ is possible (and occurs frequently in practical examples), in which case
the corresponding interval module is the zero module and can be skipped from the barcode.
We call a column that yields a zero module a \emph{zero-persistence column}.
  
We point out that the reduced matrix does not only yield the interval decomposition
$\bigoplus I_j$ that is isomorphic to $M_{(\minus ,i)}$, but it also encodes a barcode basis,
that is, an isomorphism between the two objects. Writing $g_1,\ldots,g_s$ for the generators
of $M_{(\minus,i)}$, 
every non-zero column of the reduced matrix is a linear combination
of the $g_i$'s which comes into existence at $b$, the scale of the pivot index,
and becomes trivial at $d$, the scale of the column index. Clearly, the columns
are linearly independent, and by adding suitable columns
for bars to infinity, we can complete these columns to a basis of the generators
where each basis element represents a bar. The base change from $g_1,\ldots,g_s$ to this basis
induces an isomorphism of persistence modules which
is the barcode basis in the notation of the previous section. Note that other reduction strategies
with left-to-right column additions yield the same pivot entries but possibly a different
reduced matrix and, consequently, a different barcode basis, re-stating the fact that the barcode basis
is not unique. 

Instead of performing the above procedure for every $M_{(\minus,i)}$ separately,
barcode bases for $M_{(\minus,1)},\ldots,M_{(\minus,n)}$
can be computed with one single out-of-order reduction of the matrix
for $M_{(\minus,n)}$. The idea is to group the generators and relations
of same height $i$ into $n$ ``batches'' and build the columns of $M_{(\minus,n)}$
batch-by-batch. The columns of a batch are inserted at their
column index in the final matrix, that is, potentially leaving gaps of empty
columns in the matrix that are filled by later batches. After inserting
the batch $i$, the algorithm reduces the resulting partially-filled matrix~--
this requires column operations to reduce the newly inserted columns, but potentially also
operations on old columns that can be further reduced due to the new ones.
However, it is simple to detect those columns that require an update during
the algorithm, and the accumulated cost for
all these partial reductions is not larger than if $M_{(\minus,n)}$ would be
reduced from scratch. Once the reduction of batch $n$ is finished, the matrix
encodes a barcode basis for $M_{(\minus,n)}$, and the complexity is cubic
in the input size, as for standard matrix reduction.
Note that this procedure computes the vertices of the (uncompressed) graphcode;
however we ignore zero-persistence columns by not putting a graphcode vertex for them.

With some book-keeping, the aforementioned batch-reduction
also yields the edges of the graphcode:
a column $b$ of the matrix gets
reduced in batch $i+1$ by adding other columns to it, so we obtain a linear relation of the form
\[b_{\textrm{new}}=b_{\textrm{old}}+c_1+\ldots+c_r\]
where $b_{\textrm{old}}$ is the state of column $b$ before and $b_{\textrm{new}}$ the state of $b$ after the reduction,
and $c_1,\ldots,c_r$ are columns on the left of $b$ in the current matrix (note that $r=0$ is possible
if column $b$ is not modified at batch $i+1$). Because we reduce the columns from left to right in every batch,
$c_1,\ldots,c_r$ are reduced columns and thus elements of the barcode basis of $M_{(\minus,i\plus 1)}$ and so is
$b_{\textrm{new}}$. The column $b_{\textrm{old}}$ belongs to the barcode basis of $M_{(\minus,i)}$. By rearranging to
\[b_{\textrm{old}}=b_{\textrm{new}}+c_1+\ldots+c_r,\]
we therefore express the old basis element as linear combination of the new basis elements, and in that way
determine the image of $b_{\textrm{old}}$ under the morphism $M_{(\minus,i)}\to M_{(\minus,i\plus 1)}$.
To represent that image in the graphcode, the idea is to connect the vertex of $b_{\textrm{old}}$ (at height $i$)
to the vertex of every column on the right-hand-side of the equation (at height $i+1$) with some technical conditions:
first of all, an edge is only drawn if the two columns are not zero-persistence, as otherwise, there is no graphcode
vertex for them. Moreover, with $[b,d)$ the bar of $b_{\textrm{old}}$ and $[b',d')$ the bar of $c_i$,
we draw an edge only if $[b',d')\cap [b,d)\neq\emptyset$, or equivalently, $[b',d')\lhd [b,d)$.
Indeed, otherwise, we must have $b'<d'\leq b<d$, and as discussed earlier,
there is no morphism from the interval module given by $c_i$ to the interval module given by $b_{\textrm{old}}$.
It is easy to see that computing these edges does not increase the complexity. 

The above algorithm computes a graphcode in uncompressed form. A slight adaption yields
a compressed graphcode as well:
Call a column $b$ \emph{touched} in the reduction of batch $i+1$ if $b$ is added in batch $i+1$, another column is added
to it or if $b$ is added to some other column to the right. Then, we create a graphcode vertex for $b$ at height $i+1$
only if the column was touched at height $i+1$ or $i+2$:
Observe that if a column $b$ of the matrix is not touched in batch $i+1$
we get the relation $b_{\textrm{old}}=b_{\textrm{new}}$, hence $b_{\textrm{old}}$ has only one outgoing edge.
Also, $b_{\textrm{old}}$ and $b_{\textrm{new}}$ clearly have the same birth and death value.
Furthermore, $b$ is not added to any column on the right in batch $i+1$, so $b_{\textrm{new}}$
has no further incoming edge. If $b$ is also not touched in batch $i+2$,
then $b_{\textrm{new}}$ also has one outgoing edge, and the vertex $b_{\textrm{new}}$ is superfluous
in the sense of the previous section.
This implies that ignoring such untouched columns indeed yields a compressed graphcode.

\begin{theorem}
  The graphcode algorithm computes a compressed graphcode in $O(n^3)$ time,
  and the resulting graphcode is of size $O(n^2)$.
\end{theorem}
\begin{proof}
The running time is a direct consequence of the observation that the total number
of column operations is quadratic.
All further steps in the algorithm are minor overhead ~-- 
see also \cite{russold2024graphcode}. The size bound, which also holds for the uncompressed graphcode, follows from a simple charging argument:
every edge in the graphcode is either an edge connecting two bars
represented by the same column in consecutive levels; then charge such an edge
to the starting vertex which incurs a cost of $O(n^2)$ in total,
as the number of vertices is quadratic. Or, the edge connects two different
columns in consecutive levels, say column $c_1$ at level $i$ and column $c_2$ at
level $i+1$; but then, the column addition $c_1\gets c_1+c_2$ happened
at iteration $i$, and we can charge the edge to that column
addition. Every column addition is charged only once,
and the algorithm does only
perform $O(n^2)$ such column additions.
\end{proof}

\begin{figure}
    \centering
    \includegraphics[width=0.95\linewidth,trim={0.1cm 0.1cm 0.1cm 0.1cm},clip]{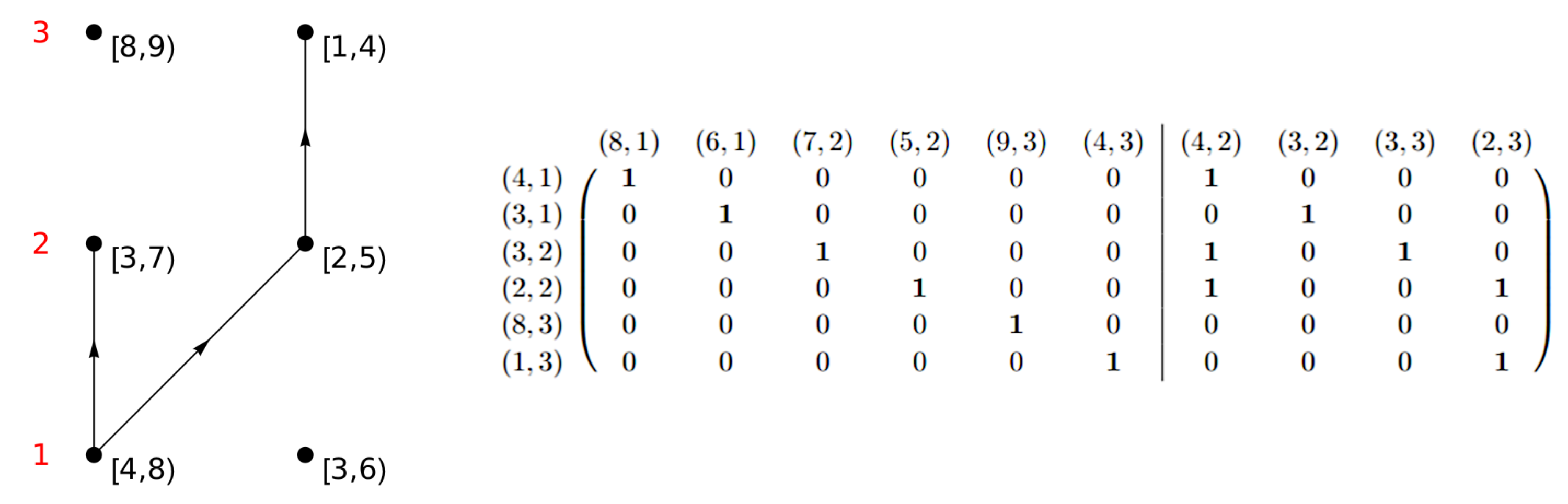}
    \caption{A presentation matrix (right) constructed from a graphcode (left).}
    \label{fig:graphcode_presentation}
\end{figure}

\subparagraph{Presentations from graphcodes.}
A graphcode
almost immediately gives rise to a presentation
that represents an isomorphic module. We describe the construction
directly for the generalized version of graphcodes. We refer to Figure~\ref{fig:graphcode_presentation} for an example.
We let $\mathcal{G}$ denote the generalized graphcode
and $\pres$ the resulting presentation.

Recall that a presentation consists of generators and relations that are both
bi-graded. We define one generator for $\pres$ at every interval contained in $\mathcal{G}$.
Precisely, for every interval $I=(b,d,h)$ of the graphcode with birth value
$b$, death value $d$ and height $h$ we place a generator $g_I$
with bi-grade $(b,h)$. Moreover, if $d\neq\infty$, we put the relation $g_I=0$
with bi-grade $(d,h)$ (capturing the fact
that $I$ is a persistence interval at height $h$).

For every generator of $\mathcal{G}$ that is not on the maximal height $n$, we add another
relation that depends on its outgoing edges in the graphcode.
Recall that in the generalized graphcode, all out-neighbors of a vertex
have the same height. For the interval $I=(b,d,h)$, let $h'$ denote this height
(if there is no outgoing edge, set $h':=h+1$). Let $J_1,\ldots,J_s$ denote
the out-neighbors of $I$ and let $g_1,\ldots,g_s$ denote the corresponding
generators. Then we put a relation $g_I=g_1+\ldots+g_s$ at bi-grade $(b,h')$.
This is well-defined because the birth value of every $J_i$ is at most $b$
by definition of the graphcode.
Informally, this relation captures the fact that the
homology class of $I$, when embedded in height $h'$, equals the sum of the homology classes of the $J_i$. 
Figure \ref{fig:graphcode_presentation} shows an example of this construction.

This construction is the inverse to the graphcode algorithm in the following sense. Let $\mathcal{G}$ be a
graphcode and let $\pres$ denote the presentation generated out of it. Then, let $\mathcal{G'}$
denote the graphcode obtained from $\pres$ by applying the graphcode algorithm.
Then $\mathcal{G}=\mathcal{G'}$; the reason is that every relation of the form $g_I=g_1+\ldots+g_s$ causes
a cascade of column operations on the column that represents $g_I$, using precisely the columns
that represent $g_1,\ldots,g_s$, resulting in exactly the right edges in the graphcode.
Hence, $\mathcal{G}$ and $\pres$ represent isomorphic persistence modules because $\pres$ and $\mathcal{G'}$ do.

It is obvious that our construction generates $\pres$ in linear time in the size of $\mathcal{G}$. However, $\pres$ is in general not minimal~-- if a minimal presentation
is required, the complexity increases to $O(n^3)$ as this is the complexity
to minimize a presentation~\cite[Thm 4.4]{fkr-compression}.
We point out that fast software for this minimization task is available
in the \textsc{mpfree} library\footnote{\url{https://bitbucket.org/mkerber/mpfree/src/master/}}; see also~\cite{bll-efficient,lw-computing}.

We remark that even if the (generalized) graphcode $\mathcal{G}$ is fully compressed, the presentation $\pres$ returned by the above procedure is in general not minimal, as we have
verified experimentally. Therefore, we cannot skip the
minimization step. Still, the question remains whether a
faster minimization procedure exists for presentations
arising from a totally compressed graphcode.

\section{Experimental evaluation}
\label{sec:experiments}
We implemented our algorithm to compute compressed graphcodes and
made it publicly
available at \url{https://bitbucket.org/mkerber/graphcode}. In its simplest form,
it takes an input chain complex or presentation in \textsc{scc2020} format\footnote{\url{https://bitbucket.org/mkerber/chain_complex_format/src/master/}}
and computes a \textsc{boost::adjacency\_list}, a standard
data structure to represent graphs.

We tested our algorithm on a large database of around 1300 minimal presentations
which is also publicly available~\cite{benchmark_repo}:
it contains minimal presentations of the instances described in~\cite{fkr-compression},
scaled up the instances described in~\cite{akll-delaunay},
and includes additional algebraic datasets described in~\cite{aida}.
This collection contains many standard constructions 
such as density-Cech bifiltrations with various choices of density,
density-Rips bifiltrations, multi-cover bifiltrations and bifiltrations
on triangular meshes. Typically, every instance appears in different sizes
and is generated five times in case the instance depends on a random process,
i.e., sampling points from a shape.

In this section, we will only report on representative outcomes for brevity.
The full log file with all instances is currently available at\\
\href{https://www.dropbox.com/scl/fo/zc5p3tbsiarnibtrnr5we/AECFnjXGYc7uRusprzeVZ24?rlkey=e60q4n27d4cvo7gj0xoewn4b9&st=22xsfwjv&dl=0}{https://www.dropbox.com/scl/fo/zc5p3tbsiarnibtrnr5we/AECFnjXGYc7uRusprzeVZ24}\\
\href{https://www.dropbox.com/scl/fo/zc5p3tbsiarnibtrnr5we/AECFnjXGYc7uRusprzeVZ24?rlkey=e60q4n27d4cvo7gj0xoewn4b9&st=22xsfwjv&dl=0}{?rlkey=e60q4n27d4cvo7gj0xoewn4b9\&st=22xsfwjv\&dl=0}. We will archive these files on a dedicated repository server upon acceptance of this article. All experiments were performed on a workstation with an Intel(R) Xeon(R)
CPU E5-1650 v3 CPU (6 cores, 12 threads, 3.5GHz) and 64 GB RAM,
running GNU/Linux (Ubuntu 20.04.2) and gcc 9.4.0.

\subparagraph{Graphcode performance.}
\begin{table}[h]
    \begin{tabular}{c|cc|ccc|cc|c}
      & & & \multicolumn{3}{c|}{Compressed} & \multicolumn{2}{c|}{Uncompressed}\\
      Type & \#Gens & \#Rels  & Vertices & Edges & time & Vertices & time & phat \\
      \hline
      \multirow{4}{*}{Density Cech}
      &   6356 &   7369 &  14845 & 10527 & 0.05 &   9.99M &   6.6  & 0.04\\
      &  12888 &  14994 &  30254 & 21636 & 0.11 &   40.6M &  27.12 & 0.08\\
      &  25973 &  30192 &  61356 & 44179 & 0.22 &    165M & 124.5  & 0.16\\
      &  52282 &  60917 & 124012 & 89874 & 0.46 & ?       & >900   & 0.33\\
      \hline
      \multirow{4}{*}{Function-Rips}
      &  14538 &  10673 &   36959 &  32901 & 0.14 & 72.8M & 50.02 & 0.07\\
      &  29206 &  21464 &   78334 &  74102 & 0.33 & ?     & >900  & 0.13\\
      &  58389 &  42896 &  166395 & 167283 & 0.58 & ?     & >900  & 0.24\\
      & 117504 &  86318 &  360671 & 388076 & 1.18 & ?     & >900  & 0.58\\
      \hline
      \multirow{4}{*}{interval}
      &  1000 &   901 &    7591 &   10391 & 0.02 &  1.24M &  0.76 & 0.01\\
      &  2000 &  1799 &   15217 &   20846 & 0.03 &  4.95M &  3.16 & 0.02\\
      &  4000 &  3609 &   30755 &   42307 & 0.06 & 19.8M  & 12.82 & 0.04\\
      &  8000 &  7186 &   61502 &   84407 & 0.13 & ?      & >900  & 0.06\\
    \end{tabular}
    \caption{Statistics for the size and computation time of graphcodes. All timings
      are in seconds, and the numbers are averaged over 5 instances of comparable
      size. 'M' stands for millions. The number of edges in the uncompressed case was always
      almost equal to the number of vertices in the uncompressed case, so it is not listed here.
    The rightmost column shows the time of phat for computing the barcode of $M_{(\minus,n)}$.}
    \label{tbl:compressed_table}
\end{table}
The advantage of compression is clearly visible in Table~\ref{tbl:compressed_table}:
the size and computation time of compressed graphcodes scales linear in most
test cases with the size of the input presentation, whereas the uncompressed
variant clearly shows a quadratic behavior~-- we remark that on those instances
where the timeout of $900$ seconds is reached, the program has occupied
the full $64$ GB of memory because of the immense size of the graph,
and the unproportional increase of runtime is due to memory swapping.
We point out that the original graphcode paper~\cite{russold2024graphcode} resolved
the issue of the large output graph by discretizing in one parameter
and only considering a small number of slices of the graphcode.
Our experiment
shows that such an approximation is not necessary with compression.

We also computed the barcode of $M_{(\minus,n)}$ using the phat library for comparison.
As visible in the last column, the running time for uncompressed graphcodes
is comparable to phat. This shows that indeed, the computation of compressed
graphcodes is almost as efficient as a single barcode computation of the
input presentation, forgetting the second component of the bi-grade.

\subparagraph{Indecomposable decompositions.}
We demonstrate that the decomposition via graphcodes helps to speed up computational tasks in multi-parameter persistence.
A natural application scenario is to decompose a module into its indecomposable summands;
we refer to that as \emph{indecomposable decomposition}.
Given some algorithm $A$ for that task that takes a presentation of the module as input,
we can compute the graphcode,  compute its connected components and
generate a minimal presentation out of every connected component with the method described in Section~\ref{sec:computation}.
Then we call $A$ on each summand. Naturally, this pre-processing with graphcodes will save time only
if the partial decomposition of the graphcode consists
of many summands, and if they are detected faster than with $A$ directly.

We tested this idea in combination with the very recent and highly-optimized \emph{aida}
decomposition algorithm from~\cite{aida}. Note that aida requires a minimal presentation as input, so we have to
minimize the presentation obtained for every connected component which typically inflicts more
costs than computing the graphcode, and (non-minimal) presentations of its components.

\begin{table}[h]
    \begin{tabular}{c|ccc|ccc|cc}
      & & & & \multicolumn{3}{c|}{Graphcode} & \multicolumn{2}{c}{Aida time}\\
      Type & \#Gens & \#Rels & \#indec  & \#comp & time & pres & on gc & on input \\
      \hline
      \multirow{4}{*}{$\begin{array}{c}\text{Multi-cover}\\ \text{hom-dim 0}\end{array}$} 
      &   4533 &   9055 &  4284 &  4176 & 0.06 & 83\% & 0.23 & 0.74\\
      &   8953 &  17895 &  8564 &  8368 & 0.12 & 42\% & 0.29 & 1.82\\
      &  18239 &  36467 & 17630 & 17351 & 0.27 & 44\% & 0.38 & 4.65\\
      &  38047 &  76083 & 36636 & 36007 & 0.76 & 46\% & 1.02 & 25.84\\
      \hline
      \multirow{4}{*}{$\begin{array}{c}\text{density-alpha}\\\text{points on $S^2$}\\\text{height function}\\\text{hom-dim 1} \end{array}$}
      &  14852 &  20302 &  11894 &  10761 & 0.29 & 60\% &  0.19 & 0.84\\
      &  32037 &  42834 &  25543 &  23169 & 0.72 & 50\% &  1.59 & 5.45\\
      &  68825 &  94486 &  54410 &  49304 & 2.40 & 54\% &  14.5 & 147\\
      & 144662 & 198627 & 113468 & 102672 & 5.23 & 52\% &  128  & >900\\
      \hline
      \multirow{4}{*}{$\begin{array}{c}\text{density-alpha}\\\text{points on $S^2$}\\\text{height function}\\\text{hom-dim 0} \end{array}$}
      &  5000 &  7451 &  4306 &  1586 & 0.28 & 85\% & 0.10 & 0.06\\
      & 10000 & 14848 &  8617 &  3253 & 0.73 & 86\% & 0.38 & 0.22\\
      & 20000 & 29609 & 17026 &  6148 & 2.83 & 91\% & 2.36 & 2.02\\
      & 40000 & 59050 & 34111 & 12402 & 10.1 & 94\% & 9.34 & 3.69
    \end{tabular}
    \caption{From left to right: Type of instance, number of generators/relations,
      number of indecomposables, number of components in graphcode, time to compute the graphcode and the minimal presentation
      of its components, percentage of running time spent on computing the minimal presentations, running time of aida for
      decomposing the graphcode output, time of aida on the input presentation. Running
      times are in seconds, and results are averages over $5$ instances of similar size.}
    \label{tbl:table_aida}
\end{table}

The results in Table~\ref{tbl:table_aida} show that in some instances, pre-processing the module
with graphcodes yields a speed-up for the indecomposable decomposition. This is, however, not the typical
case in our experiments: in many cases, the running of aida with or without graphcodes
is comparable and so, the pipeline of graphcodes and aida is slower than applying aida directly.
In extreme case, as the one shown at the bottom third of Table~\ref{tbl:table_aida}, the graphcode
part takes a long time, which is dominated by creating the presentations of the summands. 
Perhaps counterintuitively, aida runs significantly slower on the decomposed module.
We speculate that the reason is that the minimal presentation created by graphcodes, while
isomorphic to the input one, has a different structure which seems to affect
the performance of aida.
Also remarkable is that the two diametrically opposite cases presented in the middle and bottom part
of Table~\ref{tbl:table_aida} are coming from the same simplicial bifiltration; only the
homology dimension differs.
Nevertheless, the experiment shows that graphcodes can play a role in this important
algorithmic problem of multi-parameter persistence.

\subparagraph{Further experiments.}
We briefly report on further experiments that we have conducted:
we compared the file sizes of the input presentation and the resulting graphcode
(stored in a simple ascii format). The outcome is that the ratio of the graphcode size
and the presentation size is never larger than 15, and typically is less than 4.
Remarkably, there are also numerous instances where the size of the graphcode
is smaller (again, by a small constant factor). This shows that graphcodes
are indeed a reasonable alternative to represent a persistence module.

\ignore{
We investigated solving further computational tasks with graphcodes: we repeated the
experiment explained in~\cite{aida} on computing barcode templates, as used for instance
in Rivet~\cite{lw-computing} or for the computation of the exact matching distance~\cite{klo-exact,bk-asymptotic}.
Indeed, the decomposition with graphcodes can be used to reduce the number of lines required in the
constructed arrangement. Since the decomposition is only partial, the arrangement produced by graphcodes
is always at least as large as that produced by aida. We observed that in practice, while the number
of summands is typically not too different, the graphcode arrangement typically is significantly
larger, often by a factor of 10 or more. Since the size of the arrangement is quadratic in the number of lines,
the size difference incurs a much higher cost in computation, which dominates the saving we obtain
by the fact that graphcodes are usually faster to compute than indecomposable decompositions. We conclude
that without further optimizations, aida is the better choice for this problem.
}

We also pursued the idea of computing minimal presentations itself using graphcodes. Note that
our algorithm also works with taking the $k$- and $(k+1)$-dimensional simplices
of a simplicial complex as input with minor modifications.
From the graphcode, we can compute a minimal presentation using mpfree,
as described in Section~\ref{sec:computation}.
Compared to using mpfree directly, the idea is that the graphcode construction
already filters out certain simplices and simplex pairs irrelevant for the final presentation,
and it is not immediately clear how the two approaches compare.

The experiments show that, although the graphcode approach was sometimes a bit faster, the speed-up is never more than a factor of $2$.
In many cases, graphcodes are also much slower than the direct mpfree approach. Perhaps interesting
is that the graphcode presentation seems sometimes better suited for subsequent computational tasks.
For instance, we found an example where aida needed $11$ seconds for decomposing the resulting graphcode presentation,
but $160$ seconds for the mpfree presentation.
This shows once more that the graphcode view on persistence modules 
is an alternative that deserves further consideration.

\section{Deciding interval-decomposability}
\label{sec:interval_decomposability}

In this section we introduce a new criterion for interval-decomposability based on graphcodes which leads to an efficient algorithm deciding if a module is interval-decomposable. A persistence module $M\colon G(m,n)\rightarrow \mathbf{Vec}$ is interval-decomposable if it is isomorphic to a sum of interval modules $M\cong \bigoplus_{i=1}^r J_i$. By Corollary \ref{cor:sum_to_union} and Proposition \ref{prop:graphcode_from_interval} in Section \ref{sec:graphcodes}, there exists a barcode basis $\mu$ such that $\mathcal{GC}_\mu(M)$ is a disjoint union of directed paths. On the other hand, if there exists a barcode basis $\mu$ such that $\mathcal{GC}_\mu(M)$ is a disjoint union of directed paths, then, by Proposition \ref{prop:persmod_from_graphcode_union}, \ref{prop:persmod_from_path} and \ref{prop:persmod_reconstruction}, $M$ is interval-decomposable. These observations lead to the following theorem characterizing interval-decomposability.

\begin{theorem} \label{thm:interval_criterion}
A persistence module $M$ \ignore{\colon G(m,n)\rightarrow\mathbf{Vec}}is interval-decomposable if and only if there exists a barcode basis $\mu$ such that $\mathcal{GC}_\mu(M)$ is a disjoint union of directed paths. 
\end{theorem}

\begin{proof}
Suppose $M$ is interval-decomposable, i.e.\ there exist interval modules $J_i\colon G(m,n)\rightarrow\mathbf{Vec}$ such that $M\cong\bigoplus_{i=1}^r J_i$ . By Corollary \ref{cor:sum_to_union}, there exists a barcode basis such that $\mathcal{GC}_\mu(M)=\mathcal{GC}_{\text{id}}(J_1)\sqcup\cdots\sqcup \mathcal{GC}_{\text{id}}(J_r)$. By Proposition \ref{prop:graphcode_from_interval}, $\mathcal{GC}_{\text{id}}(J_i)$ is a directed path for all $1\leq i\leq r$. Conversely, assume there exists a barcode basis $\mu$ such that $\mathcal{GC}_\mu(M)=\mathcal{G}_1\sqcup\cdots\sqcup\mathcal{G}_r$ where $\mathcal{G}_i$ is a directed path for all $1\leq i\leq r$. By Proposition \ref{prop:persmod_from_graphcode_union}, $PM\big(\mathcal{GC}_\mu(M)\big)= PM(\mathcal{G}_1)\oplus\cdots\oplus PM(\mathcal{G}_r)$. By Proposition \ref{prop:persmod_from_path}, $PM(\mathcal{G}_i)$ is an interval module for all $1\leq i\leq r$. Finally, by Proposition \ref{prop:persmod_reconstruction}, $M\cong PM\big(\mathcal{GC}_\mu(M)\big)$. Therefore, $M$ is interval-decomposable.
\end{proof}

\noindent
Since the edges of $\mathcal{GC}_\mu(M)$ correspond to the matrix representations of the morphisms $(\eta^i)_{i=1}^{n\minus 1}$ in \eqref{eq:equivalence}, connecting consecutive slices, $\mathcal{GC}_\mu(M)$ is a disjoint union of directed paths if and only if $\eta^i$ has at most one non-zero entry per column and row. We say that a matrix $\eta^i$, satisfying this property is in \emph{normal form}.

If $M$ is interval-decomposable, the paths in $\mathcal{GC}_\mu(M)$ can be viewed as a generalization of bars in the one-parameter case. We can compute the interval decomposition, or in other words, the barcode of a one-parameter persistence module, represented by a sequence of matrices, by simultaneously transforming all matrices to normal form. Our new criterion for interval-decomposability in Theorem \ref{thm:interval_criterion} allows us to generalize this algorithm to the two-parameter setting.    
We introduce an algorithm that computes the interval decomposition of a two-parameter persistence module $M$ by simultaneously transforming all matrices $\eta^i$ to normal form via changes of barcode bases or decides that this is not possible.



\textbf{Valid column and row operations via change of barcode basis.} 
Our algorithm will proceed by iteratively trying to normalize the matrices $\eta^1,\ldots,\eta^{n\minus 1}$ using column and row operations. As in the case of ordinary matrices, representing linear maps, we can manipulate the matrix representation $\eta^i$, of a morphism between sums of one-parameter intervals \eqref{eq:interval_base_change}, via column and row operations, by applying certain isomorphisms $\phi_i$ and $\phi_{i\plus 1}$.
\begin{equation}\label{eq:interval_base_change}
\begin{tikzcd}
\bigoplus_{k=1}^{s_1} I^1_k \arrow[r,"\eta^1"] \arrow[d,"\phi_1","\cong"'] & \cdots \arrow[r,"\eta^{i\minus 1}"] & \bigoplus_{k=1}^{s_i} I^i_k \arrow[r,"\eta^i"] \arrow[d,"\phi_i","\cong"'] & \bigoplus_{k=1}^{s_{i\plus 1}} I^{i\plus 1}_k \arrow[r,"\eta^{i\plus 1}"] \arrow[d,"\phi_{i\plus 1}","\cong"'] & \cdots \arrow[r,"\eta^{n\minus 1}"] & \bigoplus_{k=1}^{s_n} I^n_k \arrow[d,"\phi_n","\cong"'] \\
\bigoplus_{k=1}^{s_1} I^1_k \arrow[r,"\overline{\eta}^1"] & \cdots \arrow[r,"\overline{\eta}^{i\minus 1}"] & \bigoplus_{k=1}^{s_i} I^i_k \arrow[r,"\overline{\eta}^i"] & \bigoplus_{k=1}^{s_{i\plus 1}} I^{i\plus 1}_k \arrow[r,"\overline{\eta}^{i\plus 1}"] & \cdots \arrow[r,"\overline{\eta}^{n\minus 1}"] & \bigoplus_{k=1}^{s_n} I^n_k 
\end{tikzcd}
\end{equation}
In contrast to linear maps between vector spaces the coefficients of the involved matrices represent morphisms between one-parameter intervals $I_l\rightarrow I_k$ which puts certain local constraints on the operations we can perform. 
If $I^i_k\lhd I^i_l$, there is an isomorphism $\phi_i$ that sends $I^i_r\rightarrow I^i_r$ for all $r\neq l$ and $I^i_l\rightarrow I^i_l+I^i_k$. This change of barcode basis affects the matrix representation of the morphism $\eta^i$ in the following way: We get the new matrix $\overline{\eta}^i$ by adding the $k$-th column of $\eta^i$ to the $l$-th column. A way in which morphisms between intervals behave differently than linear maps between one-dimensional vector spaces is that the composition of two non-zero morphisms $I^i_l\xrightarrow{1} I^i_k\xrightarrow{1} I^{i\plus 1}_r$ can be zero if $I^i_l\cap I^{i\plus 1}_r=\emptyset$. Therefore, after adding column $I^i_k$ to column $I^i_l$, we have to eliminate all non-zero entries in column $I^i_l$ where $I^i_l\cap I^{i\plus 1}_r=\emptyset$. If $I^{i\plus 1}_p\lhd I^{i\plus 1}_q$, a similar change of barcode basis $\phi_{i\plus 1}$ in the target of $\eta^i$, sending $I^{i\plus 1}_q\rightarrow I^{i\plus 1}_q+I^{i\plus 1}_p$, corresponds to adding the $q$-th row of $\eta^i$ to the $p$-th row.

\begin{figure}
    \centering
    \includegraphics[width=0.8\linewidth]{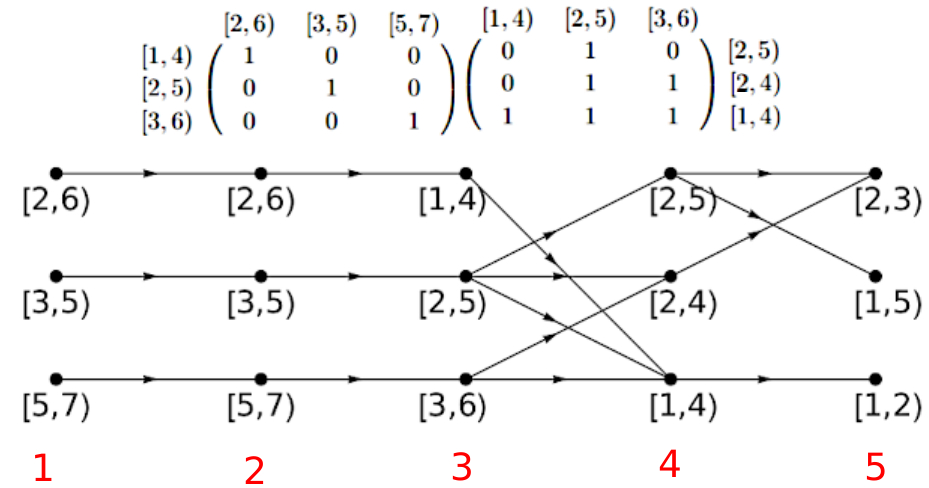}
    \caption{A graphcode that is a disjoint union of directed paths up to index $3$. At index $3$ we can not add column $1$ to column $2$ since $[2,6)\centernot{\lhd}[3,5)$ at index $2$. But we can add column $1$ to column $3$ since $[1,4)\cap[5,7)=\emptyset$.}
    \label{fig:graphcode_algo}
\end{figure}

The second constraint is that if $\eta^1,\ldots,\eta^{i\minus 1}$ in \eqref{eq:interval_base_change} are already in normal form, a column addition in $\eta^i$ will lead to a row addition in $\eta^{i\minus 1}$ which could destroy the normal form. Hence, we can only perform column operations which either don't affect the normal form of the matrices to the left or affect it in a way that can be fixed. Since $\eta^{i\minus 1}$ was in normal form before, a row addition will create at most one additional non-zero entry. This entry might be reduced by a column addition in $\eta^{i\minus 1}$ which will lead to a row addition in $\eta^{i\minus 2}$ and so on. The idea of our algorithm is that this cascade of base changes does not have to be performed. We only need to know that it could be performed. By the assumption that $\eta^1,\ldots,\eta^{i\minus 1}$ are in normal form and Theorem \ref{thm:interval_criterion}, the persistence module $M\vert_i\colon G(m,i)\rightarrow \mathbf{Vec}$, determined by $\eta^1,\ldots,\eta^{i\minus 1}$, is a sum of interval modules: $M\vert_i=\bigoplus_{l=1}^r J_l$. As depicted in Figure \ref{fig:graphcode_algo}, the intervals are exactly the paths in the corresponding graphcode up to index $i$. At index $i$ every one-parameter interval summand $I^i_k$ belongs to exactly one of those two-parameter intervals $J_{q_k}$. If we add $I^i_k=(J_{q_k})_i$ to $I^i_l=(J_{q_l})_i$ at index $i$, the additional non-zero element in $\eta^{i\minus 1}$, induced by the corresponding row addition, comes from the pivot of the predecessor $(J_{q_l})_{i\minus 1}$ on the path $J_{q_l}$ in the graphcode. The only way to eliminate it is by adding $(J_{q_k})_{i\minus 1}$ to $(J_{q_l})_{i\minus 1}$ in $\eta^{i\minus 1}$ and so on. Hence, we can add $I^i_k=(I_{q_k})_i$ to $I^i_l=(J_{q_l})_i$, without destroying the normal form to the left, if we can "add" the whole interval $J_{q_k}$ to $J_{q_l}$. We can do that if we can extend the morphism $I^i_l\xrightarrow{1} I^i_k$ to a morphism $J_{q_l}\rightarrow J_{q_k}$. This approach of reducing the matrix $\eta^i$ to normal form, while preserving the normal form up to index $i$, by valid operations on intervals, can be viewed as a direct generalization of the algorithm for general one-parameter persistence modules discussed in \cite{dey_et_al:LIPIcs.SoCG.2024.51}.

\begin{proposition} \label{prop:extending_morphisms}
Let $I,J\colon G(m,i)\rightarrow\mathbf{Vec}$ be interval modules, given as directed paths of one-parameter interval modules: 
\begin{equation}\label{eq:extending_morphisms}
\begin{tikzcd}
\cdots \arrow[r] & J_{i\minus 2} \arrow[r,"J_{i\minus 2}^{i\minus 1}"] \arrow[d,"\psi_{i\minus 2}"] & J_{i\minus 1} \arrow[r,"J_{i\minus 1}^i"] \arrow[d,"\psi_{i\minus 1}"] & J_i \arrow[d,"\psi_i"] &[-15pt] \colon &[-30pt] J \arrow[d,"\psi"] \\
\cdots \arrow[r] & I_{i\minus 2} \arrow[r,"I_{i\minus 2}^{i\minus 1}"] & I_{i\minus 1} \arrow[r,"I_{i\minus 1}^i"] & I_i & \colon & I
\end{tikzcd}
\end{equation}
such that $I_i,J_i\neq 0$ and there exists a non-zero morphism $J_i\xrightarrow{\psi_i} I_i$. Then we can extend $\psi_i$ to a morphism $\psi\colon J\rightarrow I$ if one of the following conditions holds:
\begin{enumerate}
    \item $J_{i\minus 1}=0$.
    \item $J_{i\minus 1}\cap I_{i}=\emptyset$.
    \item There exists a non-zero morphism $J_{i\minus 1}\xrightarrow{\psi_{i\minus 1}}I_{i\minus 1}$ that can be extended to a morphism $\psi\colon J\setminus J_i\rightarrow I\setminus I_i$.
\end{enumerate}
\end{proposition}

\begin{proof}
A morphism $\psi\colon J\rightarrow I$ is well-defined if $\psi_k\colon J_k\rightarrow I_k$ is well-defined for all $1\leq k\leq i$ and all the squares in \eqref{eq:extending_morphisms} commute.
\begin{enumerate}
    \item If $J_{i\minus 1}=0$, we can define $J_{k}\xrightarrow{\psi_k}I_k$ as the zero-morphism for all $1\leq k\leq i\minus 1$ which is obviously well-defined.
    \item If $J_{i\minus 1}\cap I_i=\emptyset$, the composition $\psi_i\circ J_{i\minus 1}^i=0$. Thus, we can define $\psi_{k}=0$ for all $1\leq k\leq i\minus 1$ without violating commutativity of the rightmost square in \eqref{eq:extending_morphisms}. 
    \item The existence of a non-zero morphism $\psi_{i\minus 1}\colon J_{i\minus 1}\rightarrow I_{i\minus 1}$ implies $I_{i\minus 1},J_{i\minus 1}\neq 0$ and, thus, also $I_{i\minus 1}^i,J_{i\minus 1}^i\neq 0$. By assumption $\psi\colon J\setminus J_i\rightarrow I\setminus I_i$ is a well-defined morphism up to index $i-1$. Since $\psi_i\circ J_{i\minus 1}^i$ and $I_{i\minus 1}^i\circ \psi_{i\minus 1}$ are both non-zero if and only if $J_{i\minus 1}\cap I_i\neq\emptyset$, we can extend it to index $i$ by adding $\psi_i$. \qedhere
\end{enumerate}
\end{proof}

\noindent
See Figure \ref{fig:graphcode_algo} for an example of the conditions in Proposition \ref{prop:extending_morphisms}. 
We can use this extension $\psi\colon J_{q_l}\rightarrow J_{q_k}$ of the morphism $I^i_l\xrightarrow{1} I^i_k$ to define an isomorphism of two-parameter interval modules $\phi\colon \bigoplus_{l=1}^r J_l\rightarrow \bigoplus_{l=1}^r J_l$ by $J_{q_l}\xrightarrow{(\text{id}\hspace{2pt}\psi)} J_{q_l}+J_{q_k}$ and $J_{r}\xrightarrow{\text{id}} J_r$ for $r\neq q_l$. By construction, the effect of $\phi$ is to add $I^i_k$ to $I^i_l$ in $\eta^i$ while preserving the normal form of $\eta^1,\ldots,\eta^{i\minus 1}$. Therefore, we define a column operation $I^i_k\xmapsto{+}I^i_l$ in $\eta^i$ as valid if $I^i_k\lhd I^i_l$ and one of the conditions in Proposition \ref{prop:extending_morphisms} is satisfied. The row operations on $\eta^i$ do not affect the matrices $\eta^1,\ldots, \eta^{i\minus 1}$. Hence, a row operation $I^{i\plus 1}_q\xmapsto{+}I^{i\plus 1}_p$ is valid if $I^{i\plus 1}_p\lhd I^{i\plus 1}_q$. This leads to the following algorithm which tries to transform all matrices to normal form using valid column and row operations.

In the following we will assume that the intervals $I^i_k=[b^i_k,d^i_k)$ indexing columns and rows of the matrices $\eta^i$ are ordered strictly lexicographically, i.e.\ $k<l$ implies $b^i_k<b^i_l$ or $b^i_k=b^i_l$ and $d^i_k<d^i_l$. In particular, we assume that there are no identical intervals at a certain index $i$. As a consequence $I^i_k\lhd I^i_l$ only if $k<l$ and we can only add columns from left to right and rows from bottom to top. \vspace{3pt}

\noindent
\textbf{Algorithm: } {\sc Interval-Decomposition}

\begin{itemize}
    \item \textbf{Input:} A sequence of matrices $\eta^1,\ldots,\eta^{n\minus 1}$ as in \eqref{eq:equivalence} such that the intervals $I^i_k$ indexing columns and rows are ordered strictly lexicographically.
    \item \textbf{General Procedure:} For $i=1,\ldots,n-1$. Reduce $\eta^i$ to normal form while preserving the normal form of $\eta^j$ for $1\leq j<i$ or decide that this is not possible. The algorithm maintains and updates after every step which column operations are valid.
    \item \textbf{Procedure for $\eta^i$:} For $j=1,\ldots, s_i$:
    \begin{enumerate}
        \item Reduce every possible non-zero entry of the $j$-th column of $\eta^i$ with valid column additions from the left.
        \item If the column is not zero, find the lowest row $k$ with a non-zero entry in column $j$.
        \item If row $k$ has a pivot conflict to the left, terminate and output "not interval-decomposable".
        \item Otherwise, if possible, eliminate column $j$ from bottom to top with valid additions of row $k$, while doing the corresponding column operations on $\eta^{i\plus 1}$. If this is not possible terminate and output "not interval-decomposable". 
    \end{enumerate}
    \item \textbf{Update of valid column operations:} After the matrix $\eta^{i\minus 1}$ is reduced to normal form, whose rows are the columns of $\eta^{i}$, let $I^{i}_k$ and $I^{i}_l$ be two columns of $\eta^{i}$. Then $I^{i}_k\xmapsto{+} I^{i}_l$ if $I^{i}_k\lhd I^{i}_l$ and one of the following conditions holds:
    \begin{enumerate}
        \item $I^{i}_l$ is a zero-row in $\eta^{i\minus 1}$.
        \item $I^{i}_l$ has pivot $I^{i\minus 1}_t$ in $\eta^{i\minus 1}$ and $I^i_k\cap I^{i\minus 1}_t=\emptyset$.
        \item $I^{i}_k$ has pivot $I^{i\minus 1}_s$ and $I^{i}_l$ has pivot $I^{i\minus 1}_t$ in $\eta^{i\minus 1}$ and $I^{i\minus 1}_s\xmapsto{+}I^{i\minus 1}_t$.
    \end{enumerate}
    \item \textbf{Output:} The intervals in the interval-decomposition of $M$ or "not interval-decomposable".
\end{itemize}

\noindent
We note that the algorithm can also be adapted to compressed graphcodes. The main observation is that an edge spanning multiple slices only has to be processed at its endpoints.  

Every valid operation performed by the algorithm corresponds to applying a certain isomorphism to the input module. If the algorithm does not terminate before processing all matrices, the result is a persistence module, isomorphic to the input, which has a graphcode that is a disjoint union of directed paths. By Theorem \ref{thm:interval_criterion}, these paths are the intervals in the decomposition. 
The following theorem shows that if the algorithm terminates because of a non-resolvable conflict, the input module is not interval-decomposable.

\begin{theorem} \label{thm:correctness_algorithm}
The algorithm {\sc Interval-Decomposition} is correct.
\end{theorem}

\begin{proof} 

\textbf{1.} The algorithm only performs valid column and row operations. By Proposition \ref{prop:extending_morphisms}, if a column operation $I^i_k\xmapsto{+} I^i_l$ is set to valid by the algorithm there is an isomorphism of interval modules preserving the normal form of $\eta^1,\ldots,\eta^{i\minus 1}$ while adding $I^i_k$ to $I^i_l$ in $\eta^i$ and leaving $\eta^{i\plus 1},\ldots,\eta^{n\minus 1}$ unchanged. If a row operation $I^{i\plus 1}_q\xmapsto{+} I^{i\plus 1}_p$ is valid, there is an isomorphism adding row $I^{i\plus 1}_q$ to $I^{i\plus 1}_p$ in $\eta^i$ and column $I^{i\plus 1}_p$ to $I^{i\plus 1}_q$ in $\eta^{i\plus 1}$ while leaving all other matrices unchanged. Therefore, the algorithm only applies isomorphisms to the input persistence module $M$. If the algorithm does not terminate before processing all matrices $\eta^1,\ldots,\eta^{n\minus 1}$, by Theorem \ref{thm:interval_criterion}, the resulting persistence module is a sum of interval modules isomorphic to the input $M$.

\textbf{2.} Suppose the algorithm terminates at index $i$ in step 3 but $M$ is interval-decomposable. Let $\overline{\eta}^i$ be the matrix at the point where the algorithm terminates. Then $\overline{\eta}^i$ has to be of the form:
\begin{equation} \label{eq:phi_prime_matrix}
\begin{blockarray}{ccccccc}
& \cdots & I_s^i &\cdots & I_t^i & \cdots \\
\begin{block}{c(ccc|ccc)}
\vdots & . & . & . & . & . \\ 
I_r^{i\plus 1} & . & 1 & . & 1 & . \\
\vdots & . & . & . & . & . \\
\end{block}
\end{blockarray}    
\end{equation}
where the left block is in normal form and $I_s^i$ can not be added to $I_t^i$ (termination criterion). Since $M$ is interval-decomposable there exists an isomorphism $\phi$:  
\begin{equation}\label{eq:interval_base_change_2}
\begin{tikzcd}
\bigoplus_{k=1}^{s_1} I^1_k \arrow[r,"\overline{\eta}^1"] \arrow[d,"\phi_1","\cong"'] & \cdots \arrow[r,"\overline{\eta}^{i\minus 1}"] & \bigoplus_{k=1}^{s_i} I^i_k \arrow[r,"\overline{\eta}^i"] \arrow[d,"\phi_i","\cong"'] & \bigoplus_{k=1}^{s_{i\plus 1}} I^{i\plus 1}_k \arrow[r,"\eta^{i\plus 1}"] \arrow[d,"\text{id}","\cong"'] & \cdots \arrow[r,"\eta^{n\minus 1}"] & \bigoplus_{k=1}^{s_n} I^n_k \arrow[d,"\text{id}","\cong"'] \\
\bigoplus_{k=1}^{s_1} I^1_k \arrow[r,"\theta^1"] \arrow[d,"\text{id}","\cong"'] & \cdots \arrow[r,"\theta^{i\minus 1}"] & \bigoplus_{k=1}^{s_i} I^i_k \arrow[r,"\overline{\theta}^i"] \arrow[d,"\text{id}","\cong"'] & \bigoplus_{k=1}^{s_{i\plus 1}} I^{i\plus 1}_k \arrow[r,"\eta^{i\plus 1}"] \arrow[d,"\phi_{i\plus 1}","\cong"'] & \cdots \arrow[r,"\eta^{n\minus 1}"] & \bigoplus_{k=1}^{s_n} I^n_k \arrow[d,"\phi_n","\cong"'] \\
\bigoplus_{k=1}^{s_1} I^1_k \arrow[r,"\theta^1"] & \cdots \arrow[r,"\theta^{i\minus 1}"] & \bigoplus_{k=1}^{s_i} I^i_k \arrow[r,"\theta^i"] & \bigoplus_{k=1}^{s_{i\plus 1}} I^{i\plus 1}_k \arrow[r,"\theta^{i\plus 1}"] & \cdots \arrow[r,"\theta^{n\minus 1}"] & \bigoplus_{k=1}^{s_n} I^n_k 
\end{tikzcd}
\end{equation}
such that all $\theta^j$ are in normal form and $\overline{\theta}^i$ is column reduced. We can w.l.o.g.\ assume that the columns and rows of $\theta^j$ are also ordered lexicographically. If $\overline{\theta}^i$ would not be column reduced, there would be a pivot conflict in some row that is impossible to resolve by bottom to top row additions. But, by the lexicographic order of the intervals, an isomorphism $\phi_{i\plus 1}$ can only apply bottom to top row additions on $\overline{\theta}^i$ to bring it to normal form. Since we also assume that all bars at a certain index are distinct, the normal form is unique and $\theta^j=\overline{\eta}^j$ for all $1\leq j<i$. Let $M\vert_i=\bigoplus_{k=1}^r J_k$ be the persistence module defined by $\overline{\eta}^1,\ldots,\overline{\eta}^{i\minus 1}$ which is a sum of intervals by Theorem \ref{thm:interval_criterion}. Then the restriction $\phi\vert_i\colon M\vert_i\rightarrow M\vert_i$ is a morphism of interval modules which can be factored into morphisms $J_l\rightarrow J_k$ between summands. The effect of applying $\phi\vert_i$ at index $i$ is to column reduce $\overline{\eta}^i$. By the lexicographic order, every isomorphism $\phi_i$ can only add columns from left to right. Since the left block in \eqref{eq:phi_prime_matrix} is in normal form and there is a pivot conflict between column $I^i_s$ and $I^i_t$, the only possibility to column reduce $\overline{\eta}^i$ is to add $I^i_s$ to $I^i_t$. If $J_{q_s}$ is the two-parameter interval containing $I^i_s$ and $J_{q_t}$ the interval containing $I^i_t$, then the component morphism $J_{q_t}\rightarrow J_{q_s}$ of $\phi\vert_i$ has to be non-zero at index $i$. In other words, $\phi\vert_i$ has to add $J_{q_s}$ to $J_{q_t}$. But, if $I^i_t\xrightarrow{1} I^i_s$ can be extended to a morphism $J_{q_t}\rightarrow J_{q_s}$, the column addition $I^i_s\xmapsto{+}I^i_t$ would be valid and the algorithm wouldn't terminate.

\textbf{3.} Suppose the algorithm terminates at index $i$ in step 4 but $M$ is interval-decomposable. Let $\overline{\eta}^i$ be the matrix at the point where the algorithm terminates. Then $\overline{\eta}^i$ has to be of the form:
\begin{equation} \label{eq:phi_prime_matrix_2}
\begin{blockarray}{ccccccc}
& \cdots & I_s^i &\cdots & I_t^i & \cdots \\
\begin{block}{c(ccc|ccc)}
\vdots & . & . & . & . & . \\ 
I_q^{i\plus 1} & . & . & . & 1 & . \\
\vdots & . & . & . & 0 & . \\ 
I_r^{i\plus 1} & . & . & . & 1 & . \\
\vdots & . & . & . & . & . \\
\end{block}
\end{blockarray}    
\end{equation}
where the left part of the matrix is in normal form, $I_r^{i\plus 1}$ can not be added to $I_q^{i\plus 1}$ and there is no column to the left with pivot in row $q$ that can be added to $I_t^i$. Moreover, there is no pivot conflict in row $r$ up to column $t$. Since $M$ is interval-decomposable, there exists an isomorphism as in \eqref{eq:interval_base_change_2} such that all $\theta^i$ are in normal form.

The isomorphism $\phi_i$ sends $I_t^i$ to a sum of intervals from the left without destroying the normal form in $\eta^1,\ldots,\eta^{i\minus 1}$. Since all the columns to the left of $I_t^i$ are zero or unit columns and there was no pivot conflict, the non-zero entry in row $r$ can not be eliminated. Similarly if the non-zero entry in column $t$ could be eliminated by adding a column from the left this would have happened in step 1. Thus, the non-zero entry in row $q$ can also not be eliminated and $\overline{\theta}^i$ has both these entries in column $t$. Moreover, the entry in row $r$ is still the pivot of column $t$. If the pivot would be lowered by adding a column from the left that would generate a pivot conflict but $\overline{\theta}^i$ is column reduced. 


Since $\phi_{i\plus 1}$ can only transform $\overline{\theta}^{i}$ to normal form by bottom to top row additions, it has to eliminate the non-zero entry in row $q$ with a non-zero entry in a row $p_1$ below. This non-zero entry in row $p_1$ has to be eliminated again by a non-zero entry in a row $p_2$ below. This process continues until we reach a non-zero entry in a row $p_w$ that is eliminated by the entry in row $r$. If an entry in row $p$ in column $t$ is non-zero this implies that $I_p^{i\plus 1}\lhd I^i_t$. In particular, every such interval $I^{i\plus 1}_p$ contains the starting point of the interval $I^i_t$. In this case, $I^{i\plus 1}_{p_1}\lhd I^{i\plus 1}_{p_2}$ and $I^{i\plus 1}_{p_2}\lhd I^{i\plus 1}_{p_3}$ implies $I^{i\plus 1}_{p_1}\lhd I^{i\plus 1}_{p_3}$. Hence, the relation of entanglement of intervals corresponding to rows with non-zero entries in column $t$ is transitive. Therefore, $I_{q}^{i\plus 1}\lhd I_{p_1}^{i\plus 1}\lhd\cdots\lhd I_{p_w}^{i\plus 1}\lhd I^{i+1}_{r}$ implies $I^{i\plus 1}_q\lhd I^{i\plus 1}_r$ and the row addition $I^{i\plus 1}_r\xmapsto{+}I^{i\plus 1}_q$ is valid which contradicts the termination of the algorithm. 

\end{proof}

Since the update of the valid column operations at each step takes at most quadratic time and the algorithm has to reduce every matrix $\eta^i$, the overall complexity of the algorithm is cubic in the size of all input matrices.               

\begin{theorem} \label{thm:complexity_algorithm}
{\sc Interval-Decomposition} runs in time $O\big(\text{min}(s^3,ns_\text{max}^3)\big)$ where $n$ is the number of slices, $s_\text{max}=\underset{1\leq i\leq n}{\text{max}} s_i$ and $s=\sum_{i=1}^n s_i$.
\end{theorem}

\begin{proof}
Recall the procedure to reduce the $j$-th column of $\eta^i$.
In step 1, we add every other column to the $j$-th column in the worst case.
Since $\eta^i$ is of size $s_{i\plus 1}\times s_i$,
the total cost for all those column additions is $O(s_i s_{i\plus 1})$.
Steps 2 and 3 are simply bounded by $O(s_{i+1})$.
In Step 4 we have to do $s_{i\plus 1}$ row operations on $\eta^i$ and $s_{i\plus 1}$ column operations on $\eta^{i\plus 1}$ in the worst case so this step is $O(s_i s_{i\plus 1}+s_{i\plus 1}s_{i\plus 2})$. Note that the validity of row operations can be checked in constant time. Hence, the last step is the dominating one,
and reducing the $s_i$ columns of $\eta^i$ costs
$O(s_i^2 s_{i\plus 1}+s_is_{i\plus 1}s_{i\plus 2})$.

For the procedure to update the valid column operations we have to check all pairs of columns $(I^i_k,I^i_l)$ in $\eta^i$. The check if $I^i_k\lhd I^i_l$ and the check if $I^{i\minus 1}_s\xmapsto{+}I^{i\minus 1}_t$ take constant time. So this step takes $O(s_i^2)$ time, which is dominated by the cost of reducing $\eta^i$.

Iterating over all $\eta^i$, the total cost is therefore
\[\sum_{i=1}^{n\minus 1}O(s^2_i s_{i\plus 1}+s_is_{i\plus 1}s_{i\plus 2})\]
which is both bounded by $O(s^3)$ and $O(ns_\text{max}^3)$.
\end{proof}

\begin{corollary} \label{cor:runtime_simplices}
Given a distinctly-graded presentation of size $N$, the algorithm {\sc Interval-Decomposition} decides interval-decomposability of the corresponding two-parameter persistence module in $O(N^4)$ time. 
\end{corollary}

\begin{proof}
It takes $O(N^3)$ time to compute the graphcode from the presentation. Since there are at most $N$ critical values there are at most $N$ slices. There are at most $N$ generators per slice and, thus, at most $N$ intervals per slice in the graphcode. In the notation of the previous Theorem, $n=s_\text{max}=N$, and the bound follows.     
\end{proof}

\section{Conclusion}
\label{sec:conclusion}
Graphcodes open up a new perspective on two-parameter persistence modules;
the extension to compressed graphcodes, as introduced in this paper, turn
them computationally tractable.
We consider our work only as the starting point for a deeper investigation
of the concept. For instance, instead of translating back and forth between
presentations and graphcodes, we pose the question whether subsequent
algorithms (such as module decomposition or computing other invariants) could
directly operate on the graphcode representation.

We also think that the fast partial decompositions that graphcodes provide will be
helpful as a preprocessing step in other algorithmic tasks. For that,
it is important that the graphcode splits a module into many summands.
We suspect that the ideas of Section~\ref{sec:interval_decomposability}
could lead to heuristics for graphcodes with less edges, giving rise
to finer decompositions than what the standard algorithm in Section~\ref{sec:computation} achieves. 

The interval-decomposability algorithm in its current form achieves a running
time of $O(n^4)$ which is surpassed by the independent discovery~\cite{aida} 
of a $O(n^3)$ algorithm. We strongly believe, however that our interval-decomposability algorithm will have a cubic worst-case complexity when applied on the graphcode computed in Section~\ref{sec:computation}. Indeed, most of
the column additions performed in Section~\ref{sec:interval_decomposability}
can be charged to column additions when computing the graphcode, and there are
only $O(n^2)$ of them in total. We leave this result, as well as the case that
the bars are allowed to have equal birth and death values to future work. 

A further direction to expand this work is to extend graphcodes to two-parameter persistence modules over a product of a totally ordered and a zigzag poset or even two zigzag posets. The ideas of the graphcode based interval decomposition algorithm combined with the ideas of the one-parameter zigzag algorithm could lead to an efficient decomposition algorithm for two-parameter zigzag modules.

\bibliography{bib}

\end{document}